\documentclass[10pt]{amsart}

\usepackage{amsmath,amssymb,amscd,mathrsfs,amsthm,amscd,inputenc,enumerate,amsfonts,bbold,pdfsync}
\usepackage[shortlabels]{enumitem}
\usepackage{wasysym}
\usepackage[english]{babel}

\usepackage{graphicx}
\usepackage{mathtools}

\usepackage[all]{xy}    
\usepackage{tikz-cd}
\usepackage{caption}

    \newcommand{\SStar}{{\mbox{\rm\texttt{SStar}}}}   
     \newcommand{\SStarf}{\mbox{\rm\texttt{SStar}}_{\!\mbox{\tiny\it\texttt{f}}}}  
    
    \newcommand{\SStarstab}{\overline{\mbox{\rm\texttt{SStar}}}}
    
     \newcommand{\SStarstabft}{\widetilde{\mbox{\rm\texttt{SStar}}}}
     \newcommand{\SStarsp}{\mbox{\rm\texttt{SStar}}_{\!\mbox{\tiny\it\texttt{sp}}}}
     
      \newcommand{\SStarspf}{\mbox{\rm\texttt{SStar}}_{\!\mbox{\tiny\it\texttt{f,sp}}}}
        \newcommand{\SStareabf}{\mbox{\rm\texttt{SStar}}_{\!\mbox{\tiny\it\texttt{f,eab}}}}
        \newcommand{\SStareab}{\mbox{\rm\texttt{SStar}}_{\!\mbox{\tiny\it\texttt{eab}}}}
        \newcommand{\SStarval}{\mbox{\rm\texttt{SStar}}_{\!\mbox{\tiny\it\texttt{val}}}}

   \newcommand{\stf} {\star{_{\!{_f}}}}  
  \newcommand{\stt} {\widetilde{\star}}  
    \newcommand{\stu} {\overline{\star}} 
     \newcommand{\sta} {\star_{\mbox{\it\tiny\texttt{a}}}}

\DeclareMathOperator{\Loc}{\mbox{\rm\texttt{Loc}}}

\newcommand{\overr}{{\mbox{\rm\texttt{Overr}}}}
\newcommand{\overric}{{\mbox{\rm\texttt{Overr}}}_{\!\mbox{\tiny\it\texttt{ic}}}} 
\newcommand{\overrloc}{{\mbox{\rm\texttt{Overr}}}_{\!\mbox{\tiny\it\texttt{loc}}}}
\newcommand{\overrflat}{{\mbox{\rm\texttt{Overr}}}_{\!\mbox{\tiny\it\texttt{flat}}}}

\newcommand{\spec}{{\mbox{\rm \texttt{Spec}}}}
\newcommand{\Zar}{{\texttt{Zar}}}
\newcommand{\qspec}{{\texttt{QSpec}}}
\newcommand{\qmax}{{\texttt{QMax}}}
\newcommand{\X}{{\mathbb{X}}}

\usepackage{color}

  \newcommand{\locsist}{\mbox{\rm\texttt{LS}}}
    \newcommand{\locsistf}{\mbox{\rm\texttt{LS}}_{\mbox{\it\tiny\texttt{f}} }}

 \newcommand{\insA}{\mathbb A}

\newcommand{\Kr}{{\rm Kr}}
\newcommand{\Max}{{\rm Max}}

\newcommand{\ms}{\mathscr}

\newcommand{\mf}{\mathbf}

\newcommand{\xcal}{{\boldsymbol{\mathcal{X}}}}

\newcommand{\ucal}{{\boldsymbol{\mathcal{U}}}}

    \newcommand{\FF}{\boldsymbol{\overline{F}}}
                \newcommand{\F}{\boldsymbol{F}}
                \newcommand{\f}{\boldsymbol{f}}
    \DeclareMathOperator{\chius}{\mbox{\texttt{Cl}}}

\newtheoremstyle{mio}%
	{}{} 
	{\itshape}{} 
	{\bfseries}{.}{ } 
	{#1 #2\thmnote{\mdseries~(\scshape #3)}} 
\theoremstyle{mio}
\newtheorem{teor}{Theorem}[section]
\newtheorem{cor}[teor]{Corollary}
\newtheorem{prop}[teor]{Proposition}
\newtheorem{lemma}[teor]{Lemma}

\theoremstyle{definition}

\newtheorem{ex}[teor]{\textbf{Example}}
\newtheorem{oss}[teor]{Remark}

\DeclareMathOperator{\Cl}{\mbox{\texttt{Cl}}}

  \DeclareMathOperator{\calbW}{\boldsymbol{\mathcal W}}%
   \DeclareMathOperator{\calbV}{\boldsymbol{\mathcal V}}%
      \DeclareMathOperator{\calbY}{\boldsymbol{\mathcal Y}}%

\begin{document}     

\title[]
{New distinguished classes of spectral spaces:\\ a survey}
\let\thefootnote\relax\footnote{\hskip -10 pt C.A. Finocchiaro \ $\bullet$ \ M. Fontana \ $\bullet$ \ D. Spirito \\
Dipartimento di Matematica e Fisica, Universit\`a degli Studi ``Roma Tre'', 00146 Rome, Italy. \\
e-mail: \texttt{carmelo@mat.uniroma3.it;  fontana@mat.uniroma3.it; spirito@mat.uniroma3.it.} }

\let\thefootnote\relax\footnote{\hskip -10 pt   The authors gratefully acknowledge partial support from  \sl INdAM, Istituto Nazionale di Alta Matematica.}
\let\thefootnote\relax 
\let\thefootnote\relax
\author{Carmelo A. Finocchiaro,
Marco Fontana, and 
Dario Spirito}

\date{\today}

\begin{abstract}
{In the present survey paper, we  present several new classes of Hochster's spectral spaces ``occurring in nature'', actually  in multiplicative ideal theory,  and not linked to or realized in an explicit way by prime spectra of rings.  The general setting is the space of the semistar operations (of finite type), endowed with a Zariski-like topology, which turns out to be a natural topological extension of the space of the overrings of an integral domain, endowed with a topology introduced by Zariski. One of the key tool is a recent characterization of spectral spaces,  based on  the ultrafilter topology, given in \cite{Fi}.
 Several applications are also discussed.
}
 \footnotetext{\hskip -15 pt  $\mbox{Key words:}$ \it   \dots.}
\end{abstract}

\maketitle

\section{Introduction and preliminaries}

Let $X$ be a topological space. According to \cite{ho}, $X$ is called a {\it spectral space} if there exists a ring $R$ such that $\spec({R})$, with the Zariski topology, is homeomorphic to $X$. 
Spectral spaces  can be characterized in a purely topological way:   a topological space $X$ is spectral  if and only if $X$ is T$_0$ (this means that 
for every pair of distinct points of $X$, at least one of them has an open neighborhood not containing the other),
quasi-compact, admits a basis of quasi-compact open subspaces that is closed under finite intersections, and every irreducible closed  subspace $C$ of $X$ has a (unique) generic point (i.e., there exists one point $x_C\in C$ such that $C$ coincides with the closure of this point) \cite[Proposition 4]{ho}.

 In the present survey paper, we  present several new classes of spectral spaces occurring naturally in multiplicative ideal theory. 
Before doing this, we introduce, for  convenience of the reader, some background material.

\subsection{\bf Semistar operations}

 Let $D$ be an integral domain with quotient field $K$. Let  $\FF(D)$  [respectively, $\F(D)$; $\f(D)$] be the set of all nonzero $D$--submodules of $K$ [respectively, nonzero fractional ideals; nonzero finitely generated fractional ideals] of $D$ (thus, $\f(D) \subseteq \F(D) \subseteq \FF(D)$).

A \emph{semistar operation} on $D$ is a map $\star:\FF(D)\rightarrow \FF(D)$, $E\mapsto E^\star$, such that, for every $z\in K$, $z\neq 0$, and for every $E, F\in\FF(D)$, the following properties hold: $(\mathbf{\star_1})$  $E\subseteq E^\star$; $(\mathbf{\star_2})$ $E\subseteq F$ implies $E^\star\subseteq F^\star$; $(\mathbf{\star_3})$ $(E^\star)^\star=E^\star$; $(\mathbf{\star_4})$ $(zE)^\star=z\!\cdot\! E^\star$. If $D=D^\star$, then the map $\star|_{\F(D)}:\F(D)\rightarrow\F(D)$ is called a \emph{star operation} on $D$.

Semistar operations were introduced by Okabe and Matsuda in 1994 \cite{OM2} (although this kind of operations were considered by J. Huckaba in 1988, in the setting of rings with zero-divisors \cite[Section 20]{hu}), producing a more general and flexible concept than the earlier notion of a star operations which in turn were defined by Krull \cite{Krull:1935, Krull:1936, Krull} and used, among others, by Gilmer \cite[Section 32]{gi}.

 A star operation, in Krull's original terminology, was called ``prime operation'' ({\sl  Strich-Operation} or {\sl  $^\prime$-Operation}, in German  \cite{Krull:1935,Krull:1936}). The notion of semiprime operation and the relation with  that of semistar operation has been investigated in \cite{ep-15} (see also \cite{va}). Semiprime operations include various examples of specific closures, used mainly in the Noetherian setting, the most important of which is probably tight closure, originally defined in \cite{hochster}. (See \cite{ep-12} for a survey on closure operations.)

\subsection{\bf Riemann-Zariski spaces}

Let $K$ be a field and let $A$ be any subring of $K$. Let $\Zar(K|A)$ denote the set of all the valuation domains of $K$ that contain $A$ as a subring.  In the special case where $A:=D$ is an integral domain with quotient field $K$, we simply set
 $$
 \Zar(D):=\Zar(K|D) =\{V\mid V \mbox{ is a valuation domain overring of } D\}\,.
 $$

O. Zariski in  \cite{za} introduced a topological structure on the set $Z:=\Zar(K|A)$ by taking, as a basis for the open sets, the subsets $\texttt{B}_F:=\{V\in Z\mid V\supseteq    A[F]\}$, for $F$ varying in the fami\-ly 
 of all finite subsets of $K$  (see also \cite[Chapter VI, \S 17, page 110]{zs}).
 This topology is  called \emph{the Zariski topology on $Z$} and the set $Z$, equipped with this topology (denoted also by $Z^{\mbox{\tiny{\texttt{zar}}}}$), is usually called  \emph{the  Riemann-Zariski space of $K|A$} (sometimes also called abstract Riemann surface or generalized Riemann manifold of $K|A$). 
 
 In 1944, Zariski \cite{za} proved a general result that implies the quasi-compactness of $Z^{\mbox{\tiny{\texttt{zar}}}}$, and later it was proven that  $Z^{\mbox{\tiny{\texttt{zar}}}}$ is a spectral space, in the sense of M. Hochster \cite{ho} (for the case of  the space $\Zar(D)$ see \cite[Theorem 4.1]{dofefo-87}).
 More precisely, in  \cite[Theorem
 2]{dofo-86}  (respectively, in  \cite[Corollary 3.4]{fifolo2}) the authors provide explicitly a ring $R_D$ (respectively,
 $R_{K|A}$)  having the property that $\spec( R_D)$  (respectively, $\spec( R_{K|A})$)   is canonically homeomorphic
 to $\Zar(D)$ (respectively, to $\Zar(K|A)$), both endowed with the Zariski topology (see also \cite{hk}).

 Recently in \cite{FiSp} the Zariski topology on $\Zar(D)$ was explicitly extended on the larger space $\overr(D)$ of all overrings of $D$, by taking, as a basis of open sets the collection of the sets of the type   $\overr(D[F])$,  for $F$ varying in the fami\-ly 
 of all finite subsets of $K$ (see also \cite[page 115]{zs}). Clearly, in this way, $\Zar(D)$ becomes a subspace of $\overr(D)$.

\subsection{\bf The inverse topology on a spectral space.}

Let $X$ be a topological space and let $Y$ be any subset of $X$. We denote by $\Cl(Y)$ the closure of $Y$ in the topological space $X$.  Recall  that the topology on $X$ induces a natural preorder $\leq_X$ on $X$ (simply denoted by $\leq$,  if no confusion can arise), defined  by setting $x\leq_X y$ if $y\in\Cl(\{x\})$. It is straightforward that $\leq_X$ is a partial order if and only if $X$ is a T$_0$ space (e.g., this holds when $X$ is spectral). The set
$
Y^{{\mbox{\tiny{\texttt{gen}}}}}:=  \{x\in X \mid  y\in \chius(\{x\}),\mbox{ for some }y\in Y \}
$
is called \textit{closure under generizations of $Y$}. Similarly, using the opposite order,  the set
$
Y^{{\mbox{\tiny{\texttt{sp}}}}}:= \{x\in X \mid  x\in \chius(\{y\}),\mbox{ for some }y\in Y \}
$
is called \textit{closure under specia\-li\-zations of $Y$}. We say that $Y$ is \textit{closed under generi\-zations}  (respectively, \textit{closed under specia\-li\-zations})  if $Y= Y^{{\mbox{\tiny{\texttt{gen}}}}}$ (respectively, $Y=Y^{{\mbox{\tiny{\texttt{sp}}}}}$). 
For two elements $x, y$ in a spectral space $X$, we have: 
$$
x \leq y \quad \Leftrightarrow \quad \{x \}^{{\mbox{\tiny{\texttt{gen}}}}} \subseteq \{y \}^{{\mbox{\tiny{\texttt{gen}}}}} \quad \Leftrightarrow \quad
\{x \}^{{\mbox{\tiny{\texttt{sp}}}}} \supseteq \{y \}^{{\mbox{\tiny{\texttt{sp}}}}}\,.
$$
Suppose that $X$ is a spectral space; then,  $X$ can be endowed with another topology, introduced by Hochster \cite[Proposition 8]{ho}, whose basis of closed sets is the collection of all the 
quasi-compact open subspaces of $X$. 
This topology is called \emph{the inverse topology on $X$}. For a subset $Y$ of $X$, let $\Cl^{{\mbox{\tiny{\texttt{inv}}}}}(Y)$ 
be the closure of $Y$, in the inverse topology of $X$;   we denote by $X^{{\mbox{\tiny{\texttt{inv}}}}}$ the set $X$, equipped with the inverse topology. The name given to this new topology is due to the fact that,  given $x,y\in X$,    $x\in \Cl^{{\mbox{\tiny{\texttt{inv}}}}}(\{y\})$ if and only if $y\in \Cl(\{x\})$, i.e., the partial order induced by the inverse topology is the opposite order of the partial order  induced by the given spectral topology \cite[Proposition 8]{ho}.  

By definition, for any subset $Y$ of $X$, we have
 $$
 \chius^{{\mbox{\tiny{\texttt{inv}}}}}(Y) =
\bigcap \{U   \mid  \mbox{ $U$ open and quasi-compact in } X,  \; 
U  \supseteq Y \}\,.
$$ 
In particular, keeping in mind that the inverse topology reverses the order of the given spectral topology, it follows \cite[Proposition 8]{ho} that the closure under generizations $\{x\}^{{\mbox{\tiny{\texttt{gen}}}}}$ of a singleton is closed in the inverse topology  of $X$, since   $$\{x \}^{{\mbox{\tiny{\texttt{gen}}}}}=\Cl^{{\mbox{\tiny{\texttt{inv}}}}}(\{x\})=\bigcap\{U \mid  U \subseteq X\mbox{ quasi-compact and open},\, x\in U \}.$$ 
  On the other hand, it is trivial, by the definition,  that the closure under specializations of a singleton $\{x \}^{{\mbox{\tiny{\texttt{sp}}}}}$ is closed in the given topology of $X$, since $\{x \}^{{\mbox{\tiny{\texttt{sp}}}}}= \chius(\{x\})$.

\section{Ultrafilter topology and spectral spaces}

The characterization of spectral spaces given in \cite[Proposition 4]{ho}  is often not easy to handle. In particular, it might be arduous to verify that a space is spectral using direct arguments involving
irreducible closed subspaces.

 The main result of the present section (Theorem \ref{spec-charact}) provides a  criterion  for deciding when a topological space is spectral, based on the use of ultrafilters. To introduce this statement, we need some basic and preliminary results on various topological structures that can be considered on the prime spectrum of a ring.  

\smallskip

It is well known that the prime spectrum of a commutative ring endowed with the Zariski topology  is always T$_0$,   but almost never T$_2$ nor T$_1$ (it is T$_2$ or Hausdorff only in the zero-dimensional case, cf. for instance \cite[Th\'eor\`eme 1.3]{ma}). Thus, in the general case, it is natural to look for a Hausdorff topology $\mathcal T$ on $\spec({R})$  such that the following properties are satisfied at the same time:
\begin{itemize}
\item $\mathcal T$ is finer than the Zariski topology;
\item $(\spec({R}), \mathcal T)$ is compact (i.e., quasi-compact and T$_2$, using the terminology of \cite{EGA}). 
\end{itemize}

A classical answer to the previous question is given in \cite[(7.2.11)]{EGA}, even in the more general setting of the underlying topological space of a scheme, 
by considering  the \emph{constructible topology} (see \cite{ch},  \cite[Chapter 3,  Exercises 27, 28 and 30]{am}) or   the \emph{patch topology} \cite{ho}.

As in \cite{sch-tr}, we introduce  the \emph{constructible topology}  by a Kuratowski closure ope\-rator: if  $ X$ is a spectral space,   we set, for each subset $Y$  of $X$,
$$
\begin{array}{rl}
\chius^{\mbox{\tiny{\texttt{cons}}}}(Y) \! :=  \!\bigcap \{U \!\cup\! (X\!\setminus\! V) \mid  & \hskip -5pt \mbox{ $U$ and $V$ open and quasi-compact in }  X,  \\
 & \hskip -4pt U  \!\cup \! (X \!\setminus \! V) \supseteq Y \}\,.
 \end{array}
$$
We denote by $X^{\mbox{\tiny{\texttt{cons}}}}$ the set $X$, equipped with the constructible topology.
For Noetherian spectral spaces, the clopen subsets of the constructible topology are precisely the constructible subsets after C. Chevalley \cite{ch}, i.e. the finite unions of locally closed subspaces. 
It is straighforward that the constructible topology  is a refinement of the given topology (it is the coarsest topology on $X$ for which the quasi-compact open subspaces are clopen)
 and it  is always 
Hausdorff. 
Finally, by \cite[Remark 2.2]{fifolo2}, we have 
$\chius^{\mbox{\tiny{\texttt{inv}}}}(Y)=(\chius^{\mbox{\tiny{\texttt{cons}}}}(Y))^{\mbox{\tiny{\texttt{gen}}}}$. It follows that each closed set in the inverse topology is closed under generizations  and, from \cite[Proposition 2.6]{fifolo2},  that a quasi-compact subspace $Y$ of $X$ closed for generizations is inverse-closed.  
On the other hand, the closure of a subset $Y$ in the given topology of $X$, $\chius(Y)$, coincides with   
$(\chius^{\mbox{\tiny{\texttt{cons}}}}(Y))^{\mbox{\tiny{\texttt{sp}}}}$   \cite[Remark 2.2]{fifolo2}.

\medskip
 In the following result we collect some well known classical properties of $\spec({R})$, equipped with the constructible topology. 

\begin{teor}{\rm (cf. \cite[Chapter 3, Exercises 27, 28
and 30]{am}, \cite[Proposition 5]{fo-lo-2008}, \cite[Th\'eor\`eme 2.2]{ma}, \cite[Proposition 5]{oli1} and \cite{oli2})}
Let $R$ be a ring. We denote by ${\spec}({R})^{{\mbox{\rm\tiny{\texttt{zar}}}}}$ 
(respectively, $\spec({R})^{\mbox{\rm\tiny{\texttt{cons}}}}$) the set $\spec({R})$, endowed with the Zariski topology (respectively, the constructible topology). The following properties hold.
\begin{enumerate}
\item[\rm (1)] $\spec({R})^{\mbox{\rm\tiny{\texttt{cons}}}}$ is compact, Hausdorff and totally disconnected (and, by definition, the topology is finer than the Zariski topology).
\item[\rm (2)] $\spec({R})^{\mbox{\rm\tiny{\texttt{cons}}}}=\spec({R})^{\mbox{\rm\tiny{\texttt{zar}}}}$ if and only if $R$ is zero-dimensional. 
\item [\rm (3)]  Assume that $\spec({R})^{\mbox{\rm\tiny{\texttt{zar}}}}$ is a Noetherian space. Then, a subset of  \ $\spec({R})$ is clopen 
 in $\spec({R})^{\mbox{\rm\tiny{\texttt{cons}}}}$
  if and only if it is \emph{constructible}, according to Chevalley  (see \cite{ch-1955, ch} and \cite[(2.3.11) and (2.4.1)]{EGA}) (i.e., it is a finite union of locally closed subsets of $\spec({R})^{\mbox{\rm\tiny{\texttt{zar}}}}$). 
\item[\rm (4)] Let $\{\mathbb{X}_f\mid f\in R \}$ be a collection of algebraically independent indeterminates over $R$, let $I$ be the ideal of the polynomial ring $R[\{\mathbb{X}_f\mid f\in R \}]$ generated by the set $\{f^2\mathbb{X}_f-f; f\mathbb{X}^2_f  -\mathbb{X}_f \mid f\in R\}$, and consider the ring ${\rm T}({R}):=R[\{\mathbb{X}_f\mid f\in R \}]/I$. Then, the following statements hold.
\begin{enumerate}
\item[\rm (4.a)]  ${\rm T}({R})$ is absolutely flat (or, \emph{von Neumann regular}, i.e., for each $a\in {\rm T}({R})$ there exists $x\in {\rm T}({R})$ such that
$ax^2=a$), called the \emph{absolutely flat cover} of $R$.
\item[\rm (4.b)]  The canonical embedding $\iota: R\rightarrow {\rm T}({R})$ is an epimorphism in the category of rings. 
Furthermore, $\iota$ is an isomorphism if and only if $R$ is absolutely flat. 
\item[\rm (4.c)]  The canonical continuous map $\iota^a:\spec({\rm T}({R}))^{\mbox{\rm\tiny{\texttt{zar}}}}\rightarrow \spec({R})^{\mbox{\rm\tiny{\texttt{cons}}}}$,  induced by $\iota$, is an homeomorphism. In particular, the topological space
 $\spec({R})^{\mbox{\rm\tiny{\texttt{cons}}}}$ is spectral. 
\end{enumerate}
\end{enumerate}
\end{teor}

In \cite{fo-lo-2008} a new description of $\spec({R})^{{\mbox{\tiny{\texttt{cons}}}}}$ is presented, by using a new tool: {\sl convergence by ultrafilters.} 

For the reader's convenience, we recall now some basic facts about ultrafilters (for further properties see, for example, \cite{je}). Let $\mf X$ be a nonempty set. A nonempty collection $\ms U$ of nonempty subsets of $\mf X$ is called \emph{an ultrafilter on $\mf X$} if the following axioms hold:
\begin{itemize}
\item If $Y, Z\in \ms U$, then $Y\cap Z\in \ms U$.
\item If $Y\in \ms U$ and $Y\subseteq Z\subseteq \mf X$, then $Z\in\ms U$. 
\item If $Y\subseteq \mf X$ then either $Y\in \ms U$ or $\mf X \setminus Y\in\ms U$. 
\end{itemize}

It is easy to see that, for each $x\in \mf X$, the collection $\ms U_x:=\{Y\subseteq \mf X \mid x\in Y \}$ is an ultrafilter on $\mf X$, called \emph{the trivial (or principal) ultrafilter generated by $x$}. Every finite set admits only trivial ultrafilters. The existence of nontrivial ultrafilters on infinite sets is guaranteed by the Axiom of Choice. Precisely, it is proved under ZFC that, if $\mathcal F$ is a nonempty collection of subsets of $\mf X$ with the finite intersection property, then there exists an ultrafilter $\ms U$ on $\mf X$ such that $\mathcal F\subseteq \ms U$. 

\smallskip

Now, let $R$ be a ring, let $Y$ be a nonempty subset of $\spec({R})$ and let $\ms U$ be an ultrafilter on $Y$. 
For each $f \in R$ we set $\texttt{V}(f) := \{ P \in \spec({R}) \mid f \in P\}$.
It is easy to show that the set $P_{Y,{\ms U}}:= P_{\ms U}:= \{f\in R\mid \texttt{V}(f)\cap Y\in \ms U\}$ is a prime ideal of $R$ 
 \cite[Lemma 2.4]{calota}, called \textit{the ultrafilter limit point of $Y$, with respect to $\ms U$.} 
According to \cite[Definition 1]{fo-lo-2008}, a  nonempty subset $Y$ of $\spec({R})$ is {\it ultrafilter closed} if, for any ultrafilter $\ms U$ on $Y$, we have $P_{\ms U}\in Y$.   We assume that the empty set is ultrafilter closed. The following result relates the constructible topology and the convergence   by  ultrafilters.

\begin{teor}\label{ultracons}{\rm(cf. \cite[Theorem 8]{fo-lo-2008})}
Let $R$ be a ring and let $Y\subseteq \spec({R})$. Then,  the following conditions are equivalent.
\begin{enumerate}
\item[\rm (i)] $Y$ is closed, with respect to the constructible topology.
\item[\rm (ii)] $Y$ is ultrafilter closed. 
\end{enumerate}
\end{teor}

In \cite[Section 2]{Fi}, the convergence   by  ultrafilters, presented in \cite{fo-lo-2008}, is extended in a more general setting. 
Precisely, let $\mf X$ be a nonempty set and $\mathcal F$ be a nonempty collection of subsets of $\mf X$. If  $Y$ is a nonempty subset of $\mf X$ and $\ms U$ is an ultrafilter on $Y$, we define
$$
Y_{\mathcal F}(\ms U):=\{x\in\mf X\mid[\ \forall F\in \mathcal F, x\in F\iff F\cap Y\in\ms U \ ] \}
$$
and call it \emph{the $\mathcal F$-ultrafilter limit set of $Y$, with respect to $\ms U$}. 
\begin{ex}\label{ultra-limit-point}(cf. \cite[Example 2.1(2)]{Fi})
Let $R$ be a ring, let ${\boldsymbol{\mathcal P}}$ denote the collection of the principal open subset of $\spec({R})$, i.e.,
$$
{\boldsymbol{\mathcal P}}:=\{\texttt{D}(f):=\{P \in \spec({R})\mid f\notin P \} \mid f\in R\}\,.
$$
If $\ms U$ is an ultrafilter on a subset $Y$ of $\spec({R})$, then  $Y_{\boldsymbol{\mathcal P}}(\ms U)=  \{P_{\ms U}\}$, where $P_{\ms U} $ denotes,  as before, the ultrafilter limit point of $Y$, with respect to $\ms U$.
\end{ex}

\begin{ex} 
Let $K$ be a field and let $A$ be any subring of $K$. In the space $\Zar(K|A)$, let 
$$
\boldsymbol{\mathcal B}:=\{\texttt{B}_F:= \Zar(K|A[F])\mid F\subseteq K, \ F \mbox{ finite} \}, 
$$
denote the standard basis for the open sets for the Zariski topology on  $\Zar(K|A)$.
If $Z$ is a nonempty subset of $\Zar(K|D)$ and $\ms U$ is an ultrafilter on $Z$, it is easy to show that the subset 
$$
Z_{\ms U}:=\{x\in K\mid \Zar(K|A[x])\cap Z\in \ms U \}
$$
is still a valuation domain of $K$ (cf. \cite[Lemma 2.9]{calota} and \cite[Proposition 3.1]{fifolo1}), called \emph{the ultrafilter limit point of $Z$, with respect to $\ms U$.} Then we have $Z_{\boldsymbol{\mathcal B}}(\ms U)=\{Z_{\ms U} \}$. 
\end{ex}


The next goal is to extend the notion of ultrafilter closure given for the prime spectrum of a ring in a general setting. 

Let $\mf X$ be a nonempty set, $\mathcal F$   a nonempty collection of subsets of $\mf X$, and fix a nonempty subset $Y$ of $\mf X$. We say that that $Y$ is {\it $\mathcal F$-stable under ultrafilters} if, for any ultrafilter $\ms U$ on $Y$, we have $Y_{\mathcal F}(\ms U)\subseteq Y$. 

Let ${\boldsymbol{\mathcal P}}$ be as in Example \ref{ultra-limit-point}. It is easily seen that a subset of the prime spectrum of a ring is 
${\boldsymbol{\mathcal P}}$-stable under ultrafilters if and only if it is ultrafilter closed, that is, it is closed in the constructible topology (by Theorem \ref{ultracons}).

\begin{prop}\label{basic-ultra}{\rm (cf. \cite[Propositions 2.6, 2.11, 2.13 and Theorem 2.14]{Fi})}
Let $\mf X$ be a nonempty set, $\mathcal F$ be a nonempty collection of subsets of $\mf X$. Then, the following properties hold. 
\begin{enumerate}
\item[\rm(1)] The collection of all the subsets of $\mf X$ that are stable under ultrafilters is the family of the closed sets for a topology on $\mf X$,
 called the \emph{$\mathcal F$-ultrafilter topology}. We will denote by ${\mf X}^{\mbox{\rm\tiny $\mathcal{F}$\texttt{-ultra}}}$ the set 
 $\mf X$, equipped with the $\mathcal F$-ultrafilter topo\-logy. 
\item[\rm(2)]  If $\mathcal B$ is the Boolean subalgebra of the power set of $\mf X$ generated by $\mathcal F$, then $\mathcal B$ is a collection of clopen subsets of ${\mf X}^{\mbox{\rm\tiny $\mathcal{F}$\texttt{-ultra}}}$.
\item[\rm(3)]  For each subset $Y$ of $\mf X$, the closure of $Y$ in ${\mf X}^{\mbox{\rm\tiny $\mathcal{F}$\texttt{-ultra}}}$ is the set 
$$
\bigcup\{Y_{\mathcal F}(\ms U)\mid \ms U\mbox{ ultrafilter on }Y \}.
$$
\item[\rm(4)]  The following conditions are equivalent.
\begin{enumerate}
\item[\rm(i)]  ${\mf X}^{\mbox{\rm\tiny $\mathcal{F}$\texttt{-ultra}}}$ is quasi-compact.
\item[\rm(ii)]  For any ultrafilter $\ms U$ on $\mf X$, the ultrafilter limit set ${\mf X}_{\mathcal F}(\ms U)$ is nonempty. 
\end{enumerate}
\end{enumerate}
\end{prop}

\begin{ex}(cf. \cite[Remark 2.7]{Fi}) Let ${\mf X}$ be a nonempty set.
\begin{enumerate}
\item If $\mathcal B({\mf X})$ denotes the power set of ${\mf X}$, the $\mathcal B({\mf X})$-ultrafilter topology is the discrete topology. 
\item The $\{{\mf X}\}$-ultrafilter topology is the chaotic topology (i.e., the open sets are just $X$ and $\emptyset$). 
\item Let $R$ be a ring, ${\mf X}:=\spec({R})$ and $\boldsymbol{\mathcal P}$ be as in Example \ref{ultra-limit-point}. Then, the $\boldsymbol{\mathcal P}$-ultrafilter topology is the constructible topology on ${\mf X}$ by \cite[Corollary 2.17]{Fi}. 
\end{enumerate}
\end{ex}

We apply  the previous construction when  the given set  is a topological space and the collection of subsets $\mathcal F$ is a basis for the topology. 

\begin{prop}\label{spec-ultra}{\rm (cf. \cite[Proposition 3.1]{Fi})}
Let $(X,\mathcal T)$ be a nonempty topological space and $\boldsymbol{\mathcal B}$ be a basis of open sets of $X$. Then, the following statements hold.
\begin{enumerate}[\rm(1)]
\item The $\boldsymbol{\mathcal B}$-ultrafilter topology is finer than or equal to the topology $\mathcal T$. 
\item If $(X,\mathcal T)$ is a \emph{T}$_0$ space, then $X^{\mbox{\rm\tiny $\boldsymbol{\mathcal{B}}$\texttt{-ultra}}}$ is a Hausdorff and totally disconnected space. 
\item Assume now that $(X,\mathcal T)$ is \emph{T}$_0$ and that
 $X^{\mbox{\rm\tiny $\boldsymbol{\mathcal{B}}$\texttt{-ultra}}}$ is  compact. 
 Then, the $\boldsymbol{\mathcal B}$-ultrafilter topology is the coarsest topology for which $\boldsymbol{\mathcal B}$ is a family of clopen sets. Moreover, $(X,\mathcal T)$ is a spectral space and the constructible topology on $(X,\mathcal T)$ is precisely the $\boldsymbol{\mathcal B}$-ultrafilter topology.
\end{enumerate}
\end{prop}

 Note that part (3) of the previous proposition generalizes  \cite[Theorem 8]{fo-lo-2008} and  \cite[Theorem 3.4]{fifolo1}.

By using Propositions \ref{basic-ultra}(4), \ref{spec-ultra}(3) and keeping in mind \cite[Corollary to Proposition 7]{ho}, we can 
deduce new characterizations of spectral spaces and hence new criteria, based on ultrafilters, to decide if a given topological space is spectral.

\begin{teor}\label{spec-charact}{\rm (cf. \cite[Corollary 3.3]{Fi})}
For a nonempty topological space $X$, the following conditions are equivalent. 
\begin{enumerate}[\rm (i)]
\item $X$ is a spectral space. 
\item There exists a basis $\boldsymbol{\mathcal B}$  for the open sets of $X$ such that $X^{\mbox{\rm\tiny $\boldsymbol{\mathcal{B}}$\texttt{-ultra}}}$ is a compact and Hausdorff space. 
\item $X$ is a \emph{T}$_0$ space and there is a basis $\boldsymbol{\mathcal B}$ for the open sets of  $X$ such that, for any ultrafilter $\ms U$ on $X$, the ultrafilter limit set  $X_{\boldsymbol{\mathcal B}}(\ms U)$ is nonempty. 
\item $X$ is a \emph{T}$_0$ space and there is a subbasis $\boldsymbol{\mathcal S}$ for the open sets of $X$ such that, for any ultrafilter $\ms U$ on $X$, the ultrafilter limit set  $X_{\boldsymbol{\mathcal S}}(\ms U)$ is nonempty. 
\end{enumerate}
\end{teor}

The proof of Theorem \ref{spec-charact} is not constructive, since it is based on the Axiom of Choice and some of its consequences.

As an application of Theorem \ref{spec-charact}, we now determine some new classes of spectral spaces.  The key point of the proofs resides on the existence of ultrafilter limit points. 

\begin{ex}\label{R-A-B}(cf. \cite[Proposition 3.5]{Fi})
Let $A\subseteq B$ be a ring extension, and let $X:={\bf R}(B|A)$ denote the collection of all the intermediate rings between $A$ and $B$. We can make $X$   a topological space,   by generalizing the Zariski topology introduced on the space of the overrings on an integral domain (see 1.2) and  taking as a subbasis of open sets the collection
$$
\boldsymbol{\mathcal S}:=\{{\bf R}(B|A[x])\mid x\in B \}.
$$
We claim that $X$ is a spectral space. 
It is easily seen that $X$ is T$_0$ because, if $C\neq D\in X$, we can assume, without loss of generality, that there is an element $c\in C\setminus D$, and then the open set ${\bf R}(B|A[c])$ contains $C$ and does not contain $D$. 
By Theorem \ref{spec-charact}, we have to show that, if $\ms U$ is an ultrafilter on $X$, then the ultrafilter limit set $X_{\boldsymbol{\mathcal S}}(\ms U)$ is nonempty. Consider the subset 
$$
A_{\ms U}:=\{x\in B\mid {\bf R}(B|A[x])\in\ms U \}
$$
of $B$. 
We claim that $A_{\ms U}$ is a subring of $B$. 

This follows immediately from the definition of an ultrafilter, since, if $x,y\in A_{\ms U}$ then each of the sets ${\bf R}(B|A[x-y]),\ {\bf R}(B|A[xy])$ contain ${\bf R}(B|A[x])\cap {\bf R}(B|A[y])\in \ms U$, and thus ${\bf R}(B|A[x-y]),\ {\bf R}(B|A[xy])\in\ms U$, that is, $x-y,\ xy\in A_{\ms U}$. 
Furthermore, $A_{\ms U}$ contains $A$ because, for each $a\in A$, ${\bf R}(B|A[a])=X\in\ms U$. 
Therefore, $A_{\ms U}$ is an element of $X$. The fact that $A_{\ms U}\in X_{\boldsymbol{\mathcal S}}(\ms U)$ follows immediately from the definition of $A_{\ms U}$ and thus, by Theorem \ref{spec-charact}, $X$ is a spectral space. 
\end{ex}

In particular, if $A:=D$ is an integral domain and $B:=K$ is the quotient field of $D$, we deduce from the previous example that:
\begin{cor} The space 
$\overr(D)$ of the overrings of an integral domain $D$, endowed with the Zariski topology, is a spectral space. 
\end{cor}

\begin{ex}(cf. \cite[Proposition 3.6]{Fi})
Let $A,B$ and $X$ be as in the previous example, and let $X':={\bf R'}(B|A)$ be the subset of $X$ consisting of all the subrings of $B$ that are integrally closed in $B$. We claim that, with the subspace topology induced by that of $X$, the topological space $X'$ is spectral.

It is obvious that a subbasis of open sets for the topology of $X'$ is given by the family 
$
\boldsymbol{\mathcal S'}:=\{{\bf R'}(B|A[x]) \mid x\in B\}
$.
As in the previous example, the key fact is the existence in  $X'$  of ultrafilter limit points, with respect to every ultrafilter $\ms U$ on $X'$. 
Indeed, it is not difficult to show that 
$$
A'_{\ms U}:=\{x\in B\mid {\bf R'}(B|A[x])\in\ms U\}
$$
is a subring of $B$ containing $A$ that is integrally closed in $B$. 
Thus, again by definition, the ultrafilter limit set $X'_{\boldsymbol{\mathcal S'}}(\ms U)$ is nonempty, containing $A'_{\ms U}$. Again, by Theorem \ref{spec-charact},  we conclude that $X'$ is a spectral space. 
\end{ex}


In particular, if $A:=D$ is an integral domain and $B:=K$ is the quotient field of $D$, we deduce from the previous example that:
\begin{cor} The subspace 
$\overric(D)$ of $\overr(D)$, consisting of the integrally closed overrings of an integral domain $D$, endowed with the Zariski topology, is a spectral space. 
\end{cor}


\begin{ex}
We preserve the notation of Example \ref{R-A-B}, and let $X'':=\mf L(B|A)$ be the (possibly empty) subspace of $\mf R(B|A)$ consisting of all the local rings $T$ such that $A\subseteq T\subseteq B$. A subbasis for the open sets of $X''$ is clearly the family
$$
\mathcal S'':=\{\mf L(B|A[x])\mid x\in B \}
$$
We claim that, if $X''$ is nonempty, then it is spectral. Again, we need to prove that, for any ultrafilter $\ms U$ on $X''$ the ultrafilter limit set $X''_{\mathcal S''}(\ms U)$ is nonempty. As before, it is easy to infer that $A''_{\ms U}:=\{x\in B\mid \mf L(B|A[x])\in \ms U \}\in \mf R(B|A)$. It will be immediate to conclude that $A''_{\ms U}\in X''_{\mathcal S''}(\ms U)$ if we show that $A''_{\ms U}$ is a local ring. We claim that the unique maximal ideal of $A''_{\ms U}$ is 
$$
M:=\{x\in B\mid \{T\in X''\mid x\in T\setminus U(T)\}\in\ms U \}  
$$
where, as usual, $U(T)$ denotes the set of units of a ring $T$. Thus it suffices to note that $U(A''_{\ms U})=A''(\ms U)\setminus M$ (this follows easily from definitions).
\end{ex}

In particular, if $A:=D$ is an integral domain and $B:=K$ is the quotient field of $D$, we deduce from the previous example that:
\begin{cor} The subspace 
$\overrloc(D)$ of $\overr(D)$, consisting of the local overrings of an integral domain $D$,  endowed with the Zariski topology, is a spectral space. 
\end{cor}

\section{Spaces of semistar operations}

Let $D$ be an integral domain with quotient field $K$. 
As in the star operation setting, to each semistar operation $\star$ can be associated a map   $\stf: \FF(D) \rightarrow \FF(D)$ defined by
\begin{equation*}
	E^{\stf}:=\bigcup\{F^\star\mid F\subseteq E,\ F \in\f(D)\}\,,
\end{equation*}
for every $E\in\FF(D)$. The map $\stf$ is again a semistar operation, which coincides with $\star$ on finitely generated modules; moreover, $(\stf)_{\!{_f}}=\stf$. If $\star=\stf$, we say that $\star$ is a \emph{semistar operation of finite type}. We call $\stf$ the \emph{finite-type semistar operation associated to} $\star$.

 For each $T\in \overr(D)$, the map $\wedge_{\{T\}}: \FF(D) \rightarrow \FF(D)$, defined by $E^{\wedge_{\{T\}}} := ET$, for each $E \in \FF(D)$, is an example of semistar operation of finite type on $D$, called the \emph{semistar extension to $T$}.

We denote by  $\SStar(D)$ (respectively, $\SStarf(D)$) the set of all semistar operations (respectively, semistar operations of finite type) on $D$. 
The set $\SStar(D)$ can be endowed with a natural partial order $\boldsymbol{\preceq}$ 
  which turns it into a complete lattice: if $\star_1,\star_2$ are two semistar operations, say that $\star_1\boldsymbol{\preceq}\star_2$
  if $E^{\star_1}\subseteq E^{\star_2}$  for every $E\in\FF(D)$. 
 In particular, $\stf \boldsymbol{\preceq}\star$, and $\stf$ is the biggest semistar operation of finite type smaller than $\star$. 

The infimum $\wedge_{\mathbf{\mathscr{S}}}$ of a nonempty family $\mathbf{\mathscr{S}}$ of semistar operations can be written explicitly as follows:
\begin{equation*}
	E^{\wedge_{\mathbf{\mathscr{S}}}}=\bigcap\{E^\star\mid\star\in\mathbf{\mathscr{S}}\},  \;\;\; \mbox{ for each } E\in\FF(D)\,.
\end{equation*}

 In particular, if $\boldsymbol{\mathcal{T}}$ is a nonempty family of overrings of $D$, then the infimum of the family of semistar operations $\{\wedge_{\{T\}}\mid T\in\boldsymbol{\mathcal{T}}\}$ is denoted by $\wedge_{\boldsymbol{\mathcal{T}}}$.

On the other hand, there is not a general explicit formula for the supremum $\vee_{\mathbf{\mathscr{S}}}:=\bigwedge\{\sigma\in\SStar(D)\mid\star \boldsymbol{\preceq}\sigma\text{~for all~}\star\in\mathbf{\mathscr{S}}\}$, although, if $\mathbf{\mathscr{S}}\subseteq\SStarf(D)$, then
\begin{equation}\label{eq:sup}
	E^{\vee_{\mathbf{\mathscr{S}}}}=\bigcup\{E^{\star_1\circ\star_2\circ\cdots\circ\star_n}\mid \star_1,\ldots,\star_n\in\mathbf{\mathscr{S}}\}
\end{equation}
where $\star_1\circ\star_2\circ\cdots\circ\star_n$ denotes the usual composition of functions   (see \cite[p.1628]{an} and  \cite[Lemma 2.12]{FiSp}).

\medskip

A nonzero ideal $I$ of $D$ is called a \emph{quasi-$\star$-ideal} if $I=I^\star\cap D$. A \emph{quasi-$\star$-prime} is a quasi-$\star$-ideal which is also a prime ideal; the set of all quasi-$\star$-prime ideals of $D$ is denoted by $\qspec^\star(D)$. The set of maximal elements in the set of proper quasi-$\star$-ideals of $D$ (ordered by set-theoretic inclusion) is denoted by  $\qmax^\star(D)$, and it is a subset of $\qspec^\star(D)$. 
 By Zorn's Lemma, it is easy to show that if $\star$ is a semistar operation of finite type then $\qmax^\star(D) \neq \emptyset$.
 If every quasi-$\star$-ideal is contained in a quasi-$\star$-prime, then $\star$ is said to be \emph{quasi-spectral} or \emph{semifinite}. Every operation of finite type  is not only quasi-spectral, but it has the stronger property that every quasi-$\star$-ideal is contained in a maximal quasi-$\star$-ideal. Note that a semistar operation $\star$ may be quasi-spectral even if $\qmax^{\star}(D)$ is empty (see \cite[Remark 5.6]{FiSp} for an example).

 A semistar operation $\star$ is called \emph{spectral} if there is a nonempty subset $Y\subseteq\spec(D)$ such that $\star=\wedge_{\boldsymbol{\mathcal{L}}(Y)}$, where $\boldsymbol{\mathcal{L}}(Y):=\{D_P\mid P\in Y\}$. We set  $s_Y:=\wedge_{\boldsymbol{\mathcal{L}}(Y)}$ and we call $s_Y$ {\it the spectral semistar operation associated to} $Y\subseteq\spec(D)$.

A \emph{semistar operation} $\star$ is called \emph{stable} if $(E\cap F)^\star=E^\star\cap F^\star$ for every pair  $E,F\in\FF(D)$.

\begin{oss} \label{spectral}
 Every spectral semistar operation is quasi-spectral (or semifinite) by \cite[Lemma 1.4(5)]{fohu} and 
every spectral semistar operation, or more generally every operation induced by a family of $D$-flat overrings of $D$, is stable. However,  the converse does not hold in general   \cite[Section 3, page 441]{heinz-roit}, but if $\star$ is a stable semistar operation then $\star$ is spectral if and only if it is quasi-spectral (see \cite[Theorem 4]{an-overrings} and \cite[Theorem 4.12(3)]{fohu}).   In particular, a stable semistar operation of finite type is spectral.
\end{oss}

\medskip

In \cite{FiSp}, the set $\SStar(D)$ was endowed with a topology (called the \emph{Zariski topology}) by declaring open the sets of the form
\begin{equation*}
	\mbox{\texttt{V}}_E:=\{\star\in\SStar(D)\mid 1\in E^\star\}\,,
\end{equation*}
for $E \in \FF(D)$.
This topology makes $\SStar(D)$ into a quasi-compact, T$_0$ space with a unique closed point (the identity semistar operation $\mbox{\it\texttt{d}}_D$) and a generic point (the trivial semistar extension $\wedge_{\{K\}}$). In particular, $\SStar(D)$ is never T$_1$ (nor T$_2$) unless $D=K$.

\begin{prop}\label{immersione}
	Let $D$ be an integral domain,   let $\overr(D)$ and $\SStarf(D)$ be endowed with their Zariski topologies,  and let $\iota:\overr(D)\rightarrow \SStarf(D)$ be the injective map defined by $\iota(T) := \wedge_{\{T\}}$, for each $T \in \overr(D)$.  Then, the following statements hold. 
	\begin{enumerate}
		\item[\rm(1)] The map  $\iota$ is a topological embedding \cite[Proposition 2.5]{FiSp}. 
		\item[\rm(2)] The mapping $\pi:\SStarf(D)\rightarrow \overr(D)$,  defined by  $\pi(\star):=  D^\star$, for each $\star \in\SStarf(D)$,  is a continuous surjection such that $\pi\circ \iota$ is the identity of $\overr(D)$. In other words, $\pi$ is a topological retraction.
	\end{enumerate}
\end{prop}
Note that part (2) of the previous proposition follows from the fact that, for each subbasic open set  $\texttt{B}_x := \overr(D[x])$ of $\overr(D)$, we have $\pi^{-1}(\texttt{B}_x) = \{ \star \in \SStarf(D) \mid D[x] \subseteq D^\star\} =
\{ \star \in \SStarf(D) \mid 1 \subseteq (x^{-1}D)^\star \} = \texttt{V}_{x^{-1}D}$. 

\medskip

The following result relates the quasi-compactness of a collection of semistar operations on the same integral domain with the finite type property of their infimum.

\begin{prop}\label{comp-tipofinito} \rm{(cf. \cite[Proposition 2.7 ]{FiSp})}
	Let $D$ be an integral domain and let $\mathscr S$ be a quasi-compact subspace of $\SStarf(D)$. Then, $\wedge_{\mathscr S}$ is of finite type. 
\end{prop}

\begin{oss}
Let $\mathscr S$ be a subset of $\SStar(D)$ and set $\mathscr S_{\!{_f}}:= \{ \stf \mid \star \in \mathscr S\}$. Consider the following properties:
\begin{itemize}
\item[\rm(a)\;]  $\mathscr S$ is quasi-compact in $\SStar(D)$;
\item[\rm(b)\;]  $\mathscr S_{\!{_f}}$ is  quasi-compact in $\SStarf(D)$;
\item[\rm({c})\;]    $\wedge_{\mathscr S_{\!{_f}}}$ is a semistar operation of finite type;
\item[\rm({d})\;]    $\wedge_{\mathscr S_{\!{_f}}} = (\wedge_{\mathscr S})_{\!{_f}}$.
\end{itemize}
Then (a) $\Rightarrow$ (b) $\Rightarrow$ (c) $\Leftrightarrow$ (d). 

In fact, it is straightforward that (a) $\Rightarrow$ (b) (see also Proposition \ref{prop:retraction}).  
By Proposition \ref{comp-tipofinito}, (b) $\Rightarrow$ ({c}).
For ({c}) $\Rightarrow$ ({d}), note that in general  
$\wedge_{\mathscr S_{\!{_f}}} \leq \wedge_{\mathscr S}$ 
and  $(\wedge_{\mathscr S})_{\!{_f}} \leq \wedge_{\mathscr S_{\!{_f}}}$.   
 The conclusion follows from the fact that, under ({c}), $ (\wedge_{{\mathscr S}_{\!{_f}}})_{\!_{\!{_f}}}  = \wedge_{{\mathscr S}_{\!{_f}}}$.
 Finally, (d) $\Rightarrow$ ({c}) is trivial.
\end{oss}

 Since, for each overring $T$ of an integral domain $D$, the semistar operation $\wedge_{\{T\}}$ is of  finite type, we get the following result, just by applying Propositions \ref{immersione} and \ref{comp-tipofinito}.
 
\begin{cor}\label{qcover}
{\rm (cf. \cite[Corollary 2.8]{FiSp})}
	Let $D$ be an integral domain and let $\boldsymbol{\mathcal{T}}$ be a quasi-compact subspace of $\overr(D)$. Then $\wedge_{\boldsymbol{\mathcal{T}}}$ is of finite type. 
\end{cor}
In particular, the previous corollary applies when $\boldsymbol{\mathcal T}$ is {\it locally finite},    i.e., if every nonzero element of $D$ is nonunit in finitely many overrings of the family $\boldsymbol{\mathcal{T}}$ \cite[Corollary 2.10]{FiSp}.
However, the finite type property of a semistar operation $\wedge_{\boldsymbol{\mathcal{T}}}$, induced by a collection $\boldsymbol{\mathcal{T}}$ of overrings, does not imply the quasi-compactness of $\boldsymbol{\mathcal{T}}$, as the following example shows.   This example provides a negative answer to the Conjecture in \cite[page 214]{FiSp}.

\begin{ex}
Let $k$ be a field, let $\X$ be an indeterminate over $k$,  let $D:=k[\![\X^4,\X^5,\X^6,\X^7]\!]=k+\X^4k[\![\X]\!]$  and let $K:= k(\!(\X)\!)$.
	 Since $D$ is    Noetherian and a conductive domain (i.e., $(D:T) \neq (0)$ for each $T \in\overr(D)$ with $T\neq K$, see \cite[Theorem 1]{badofo}), $\FF(D)=\F(D) \cup\{K\}= \f(D)\cup\{K\}$, and thus every semistar operation on $D$ is of finite type. For every $\alpha\in K$, consider the ring $T_\alpha:=D[\X^2+\alpha \X^3]= k+(\X^2+\alpha \X^3)k+\X^4k[[\X]] $, and, for every $A\subseteq k$, let $\boldsymbol{\mathcal{T}}_{\!A}:=\{T_\alpha\mid\alpha\in A\}$. Then,   as observed above,  the semistar operation $\wedge_{\boldsymbol{\mathcal{T}}_{\!A}}$ is of finite type.
	  However, if $A$ is infinite (so, for example, if $k$ is infinite and $A=k$), then ${\boldsymbol{\mathcal{T}}_{\!A}}$ is not quasi-compact. Indeed, the open cover $\{\overr(T_\alpha)\mid \alpha\in A\}$ of  $\boldsymbol{\mathcal{T}}_{\!A}$ in $\overr(D)$ has no finite subcovers, since $\overr(T_\alpha)\cap {\boldsymbol{\mathcal{T}}_{\!A}}=\{T_\alpha\}$.
\end{ex}


The following example shows how to use Corollary  \ref{qcover}  for establishing the failure of quasi-compactness for some distinguished subspaces of  $\overr(D)$.


\begin{ex}
	Let $D$ be a Noetherian domain of dimension $\geq 2$, and let  $\boldsymbol{\mathcal{D}}$  be the set of Noetherian valuation overrings of $D$, i.e., the union of $\{K\}$ with the set of discrete valuation overrings of $D$. 
	If $I$ is a proper ideal of $D$, then $I^{\wedge_{\boldsymbol{\mathcal{D}}}}=I^{\mbox{\it\tiny\texttt{b}}}$, where ${\mbox{\it\texttt{b}}}:=\wedge_{\Zar(D)}$ (see, for example, \cite[Proposition 6.8.4]{swhu}, after noting that the terminology used therein is slightly different).
	In particular, the same holds for every $F \in \f(D)$, so that $\left(\wedge_{\boldsymbol{\mathcal{D}}}\right)_f={\mbox{\it\texttt{b}}}$. 
	However, if $W\in\Zar(D)\setminus \boldsymbol{\mathcal{D}}$ 
	 (for example, if $\dim(W)\geq 2$,
	where the existence of such a $W$ 
	 is guaranteed by \cite[Corollary 19.7]{gi}),
	  then $W$ is contained in (at most) one element $V$ of $\boldsymbol{\mathcal{D}}$,  
	  so that $WV= V$, while $WV'=K$ for each $V'\in \boldsymbol{\mathcal{D}}$, $V'\neq V$. 
	  Hence, $W^{\wedge_{\boldsymbol{\mathcal{D}}}}\neq W$, 
	  while $W^{\mbox{\it\tiny\texttt{b}}}=W$ and  thus, $\wedge_{\boldsymbol{\mathcal{D}}} \neq {\mbox{\it\texttt{b}}}$. 
	  Therefore, $\wedge_{\boldsymbol{\mathcal{D}}}$ is not of finite type, and so ${\boldsymbol{\mathcal{D}}}$ 
	  is not a quasi-compact subset of $\overr(D)$ (or of $\Zar(D)$).
\end{ex}

\begin{teor}\label{prop:insfinss-spectral}{\rm(cf. \cite[Theorem 2.13]{FiSp})}
	Let $D$ be an integral domain. Then, $\SStarf(D)$ is a spectral space.
\end{teor}

The proof uses Theorem \ref{spec-charact}, so it is not constructive. 
 However, if  $A$ is a ring such that $\spec(A)\simeq\SStarf(D)$, we can  assume that:

 (a) $A_{\texttt{red}}$ (the reduced ring associated to $A$) is an integral domain (since $\SStarf(D)$ has a unique generic point), 
 
 (b)  $A_{\texttt{red}}$ (and $A$) is local (since $\SStarf(D)$ has a unique closed point), \;\; and 
 
 ({c})   $\dim(A) = \dim(A_{\texttt{red}}) \geq |\spec(D)|-1$ (see the following    Propositions \ref{homeo} and \ref{closed-inverse}).

\medskip
On the other hand, since the proof  of Theorem \ref{spec-charact} uses in a crucial way the characterization \eqref{eq:sup} of the supremum of a family of finite-type semistar operations, it cannot readily be adapted to $\SStar(D)$ and so, up to now, we do not  know whether $\SStar(D)$ is  a spectral space.
\medskip

We denote by $\SStarstab(D)$ (respectively, $\SStarstabft(D)$) the subset of $\SStar(D)$ consisting of all stable semistar operations (respectively, all stable semistar operations of finite type).

\begin{oss} \label{sp-st}
  (a) If we set $\SStarsp(D) := \{\star \in \SStar(D) \mid \star \mbox{ is spectral}\}$ 
(respectively, $\SStarspf(D) := \{\star \in \SStarf(D) \mid \star \mbox{ is spectral}\}$), 
then  by Remark \ref{spectral} $\SStarsp(D) \subseteq  \SStarstab(D)$, and the inclusion might be proper. However,  in the finite type case, we have equality \cite[Proposition 4.23(2)]{fohu}, i.e., $$
\SStarspf(D)=  \SStarstab(D) \cap \SStarf(D)=\SStarstabft(D).
$$

(b)  Let  $\Loc(D)$ and $\overrflat(D)$ be, respectively, the set of localizations of $D$ and the set of $D$-flat overrings of $D$ (and so $\Loc(D) \subseteq \overrflat(D)$).
We observe that the topological  embedding $\iota: \overr(D)\hookrightarrow  \SStarf(D)$, considered in Proposition \ref{immersione}(1),  restricts to a topological embedding $\iota_{\mbox{\it \tiny\texttt{Loc}}}:  \Loc(D) \hookrightarrow \SStarstabft(D)$ (or to a topological embedding $\iota_{\mbox{\it\tiny\texttt{flat}}}:\overrflat(D)\hookrightarrow \SStarstabft(D)$).

 On the opposite side, the map $\pi:\SStarf(D)\rightarrow \overr(D)$  (Proposition \ref{immersione}(2)) does not always restrict to a map $\SStarstabft(D)\rightarrow\overrflat(D)$, since not all intersection of localizations of $D$ are $D$-flat (see for instance   \cite[Section 3, page 441]{heinz-roit}).
\end{oss}

\medskip

Given a semistar operation $\star$ on $D$, we can always associate to $\star$ two semistar operations $\overline{\star}$ and $\widetilde{\star}$ on $D$ defined as follows: for each $E\in\FF(D)$,
\begin{equation*}
	\begin{array}{rl} E^{\overline{\star}}:= & \bigcup \{(E:I)\mid I \text{~nonzero ideal of~} D \text{~such that~} I^\star =D^\star\},\\
		E^{\tilde{\star}}:= & \bigcup \{ (E:J)\mid J \text{~nonzero finitely generated ideal of~} D \\
		& \hskip 70pt \text{~such that~}J^\star =D^\star\}.
	\end{array} 
\end{equation*}
It is easy to see that $\stt\boldsymbol{\preceq}\stu\boldsymbol{\preceq}\star$ and, moreover, that $\stu$ (respectively, $\stt$) is the largest stable (respectively, stable of finite type) semistar operation that precedes $\star$, called  \emph{ the stable} (respectively, the \emph{the finite type stable}) \emph{semistar operation associated to $\star$}. Therefore, $\star$ is stable (respectively, stable of finite type) if and only if $\star=\stu$ (respectively, $\star=\stt$) \cite[Proposition 3.7, Corollary 3.9]{fohu}. Note that, for each semistar operation $\star$, we always have $\widetilde{\star} = s_Y$, where $Y=\qmax^{\stf}(D)$  (cf. \cite[page 182, Proposition 4.3]{fohu}, \cite[Proposition 3.4(4)]{fo-lo-2003}, \cite[Remark 10]{folo-2006} and, for the star operation case, \cite[Corollary 2.10]{ac}).

\begin{prop}\label{prop:retraction}{\rm (cf.  \cite[Proposition 4.1]{FiFoSp-4} and \cite[Proposition 2.4]{FiSp})}
	Let $\Phi_{\mbox{\it\tiny \texttt{f}}}:\SStar(D)\rightarrow\SStar(D)$ (respectively, $\overline{\Phi}:\SStar(D)\rightarrow\SStar(D)$; $\widetilde{\Phi}:\SStar(D)\rightarrow \SStar(D)$) be the map defined by  $\star\mapsto\stf$ (respectively, $\star\mapsto\stu$; $\star\mapsto\stt$). Then:
	\begin{enumerate}
		\item[\rm(1)] The images of $\Phi_{{\mbox{\it\tiny \texttt{f}}}}$, $\overline{\Phi}$ and $\widetilde{\Phi}$ are, respectively, $\SStarf(D)$, $\SStarstab(D)$ and $\SStarstabft(D)$. 
		\item[\rm(2)] The maps  $\Phi_{\mbox{\it\tiny \texttt{f}}}$, $\overline{\Phi}$ and $\widetilde{\Phi}$ are continuous in the Zariski topology.
		\item[\rm(3)] The maps $\Phi_{{\mbox{\it\tiny \texttt{f}}}}$, $\overline{\Phi}$ and $\widetilde{\Phi}$ are topological retractions of $\SStar(D)$ onto their respective images.
	\end{enumerate}
\end{prop}

Another point of similarity between finite type, stable and spectral operations is given by the open sets needed to generate the Zariski topology,  induced by the Zariski topology of $\overr(D)$. 
Indeed, if $\star$ is of finite type, let $E \in \FF(D)$, and let $\star\in \texttt{V}_E$, that is, $1\in E^\star$, then there is a finitely generated submodule $F\subseteq E$ such that $1\in F^\star$, so that $\star\in \texttt{V}_F$; it follows that
\begin{equation*}
	\texttt{V}_E\cap\SStarf(D)=\bigcup\{\texttt{V}_F\cap\SStarf(D)\mid F\subseteq E,\, F\in\f(D)\}
\end{equation*}
and thus $\{\texttt{V}_F\cap\SStarf(D)\mid F\in\f(D)\}$ is a subbasis for the Zariski topology on $\SStarf(D)$. 
Similarly, if $\star$ is stable, then $1\in E^\star$ if and only if $1\in E^\star\cap D^\star=(E\cap D)^\star$. Therefore, the Zariski topology on $\SStarstab(D)$ is generated by the $\texttt{V}_I\cap\SStarstab(D)$, as $I$ ranges among the integral ideals of $D$. 
The same reasoning shows that $\{V_J\cap\SStarstabft(D)\mid J\subseteq D,\, J\in\f(D)\}$ is a subbasis for the Zariski topology on $\SStarstabft(D)$.
 This implies that stable semistar operations are completely determined by their action inside the ring.  In particular, if $\ast: \F(D) \rightarrow \F(D)$ is a stable {\sl star} operation, then there is a unique stable {\sl semistar} operation $\hat{\ast}: \FF(D) \rightarrow \FF(D)$ such that $\hat{\ast}|_{\F(D)}=\ast$.

 \begin{oss}
Note that  the  subbasic open sets $ \overline{\texttt{U}}_I := {\texttt{V}}_I \cap \SStarstab(D) = \{\star\in \SStar(D)\mid 1\in I^\star\} \cap \SStarstab(D)$ (respectively,    $ \widetilde{\texttt{U}}_I := {\texttt{V}}_I \cap \SStarstabft(D) = \{\star\in \SStar(D)\mid 1\in I^\star\} \cap \SStarstabft(D)$)
    of $\SStarstab(D)$ (respectively, of $\SStarstabft(D)$), where $I$ is an ideal of $D$,  form a basis of $\SStarstab(D)$ (respectively, $\SStarstabft(D)$), since   $\overline{\texttt{U}}_{I'} \cap \overline{\texttt{U}}_{I''}  = \overline{\texttt{U}}_{I'\cap I''}$ (respectively, 
    $\widetilde{\texttt{U}}_{I'} \cap \widetilde{\texttt{U}}_{I''}  = 
    \widetilde{\texttt{U}}_{I'\cap I''}$), for all 
    $I'$  and $I''$ ideals of $D$.
    
     On the other hand, when considering  finitely generated ideals $J$ of $D$, in general the $ \widetilde{\texttt{U}}_J $'s do not form a basis for the open sets in $\SStarstabft(D)$, since $\widetilde{\texttt{U}}_{J'} \cap \widetilde{\texttt{U}}_{J''}  = \widetilde{\texttt{U}}_{J'\cap J''}$,  and  $J'\cap J''$ is not necessarily  finitely generated, even if
    $J'$  and $J''$  are finitely generated ideals of $D$.
\end{oss}


 \medskip

 Besides the Zariski topology, we can also endow $\SStar(D)$ with possibly weaker topologies induced by the sets considered in the above paragraph,.

\begin{prop} {\em{(cf. \cite[Proposition 2.1 and Remark 2.2]{FiSp})}} 
	Preserve the notation of Propostion \ref{prop:retraction}, and endow $\SStar(D)$ with the topology generated by 
	$\{\mbox{\rm\texttt{V}}_F\mid F\in\f(D)\}$
	(respectively, $\{\mbox{\rm\texttt{V}}_I\mid I \mbox{ ideal in } D\}$; 
	$\{\mbox{\rm\texttt{V}}_J\mid J\subseteq D,\  J\in\f(D)\}$). 
	Then, $\Phi_{{\mbox{\it\tiny \texttt{f}}}}$ (respectively, $\overline{\Phi}$; $\widetilde{\Phi}$) is the Kolmogoroff quotient of $\SStar(D)$ onto $\SStarf(D)$ (respectively, $\SStarstab(D)$; $\SStarstabft(D)$),   i.e., it is the canonical map to the quotient by the equivalence relation of ``topological indistinguishability''  (where two points of a topological space are topologically indistinguishable if they have exactly the same neighborhoods). 
	\end{prop}

  Let   $Y \subseteq \spec(D)$ be a nonempty set defining a spectral semistar operation. Then  its closure, in the inverse topology  (denoted by $\Cl^{\tiny\texttt{inv}}(Y)$, see (1.3)),  provides some useful information about $s_Y$.

\begin{prop}\label{spettrali} {\em{(cf. \cite[Corollaries 4.4 and 5.2, Proposition 5.1]{FiSp}   and \cite[Lemma 4.2 and Remark 4.5]{fohu}})}
	Let $D$ be an integral domain and   let $Y$ and $Z$ be two nonempty subsets of $\spec(D)$. The following statements hold.  
	\begin{enumerate}
	\item[\rm (1)]   $s_Y = s_Z$ if and only if $Y^{\mbox{\rm\tiny\texttt{gen}}} = Z^{\mbox{\rm\tiny\texttt{gen}}}$.
		\item[\rm (2)] $s_Y$ is of finite type if and only if $Y$ is quasi-compact. 
		
		\item[\rm (3)]  $\widetilde{s_Y}=\widetilde{s_Z}$ if and only if   ${\mbox{\rm$\Cl$}}^{\mbox{\tiny\rm\texttt{inv}}}(Y) ={\mbox{\rm$\Cl$}}^{\mbox{\tiny\rm\texttt{inv}}}(Z)$.		
		\item[\rm (4)] $\widetilde{s_Y}=s_{{\mbox{\tiny\rm$\Cl$}}^{\mbox{\tiny\rm\texttt{inv}}}(Y)}$.
			\end{enumerate}
\end{prop}

 Note that, in general, $(s_Y)_{\!_{f}}$ is quasi-spectral but not spectral, and it is spectral if and only if $(s_Y)_{\!_{f}}$ is stable \cite[Proposition 4.23(2)]{fohu}. In other words, it is possible that $\widetilde{s_Y} \lneq (s_Y)_{\!_{f}}$  (see \cite[Remark 5.3]{FiSp} and \cite[page 2466]{ac}) and thus it is not true in general that $ (s_Y)_{\!_{f}} = s_{{\mbox{\tiny\rm$\Cl$}}^{\texttt{inv}}(Y)}$.

\medskip

The following result provides control of the infimum and the supremum of a family of spectral operations:

\begin{lemma}\label{prop:infsup}  {\rm (cf.   \cite[Lemma 4.3]{FiFoSp-4})}
	Let $\mathscr{D}$ be a nonempty set of spectral semistar operations on an integral domain $D$. For each spectral semistar operation $\star$, set $\Delta(\star):={\mbox{\rm\qspec}}^\star(D)$. Then, the following statements hold.
	\begin{enumerate}[\rm (1)]
		\item\label{prop:infsup:inf} $\wedge_{\mathscr{D}}$ is spectral with $\Delta(\wedge_{\mathscr{D}})=\bigcup\{\Delta(\star)\mid\star\in\mathscr{D}\}$.
		\item\label{prop:infsup:sup} If $\vee_\mathscr{D}$ is quasi-spectral, then it is spectral with $\Delta(\vee_{\mathscr{D}})=\bigcap\{\Delta(\star)\mid\star\in\mathscr{D}\}$.
	\end{enumerate}
\end{lemma}

Note that the hypothesis that $\vee_\mathscr{D}$ be quasi-spectral in point \ref{prop:infsup:sup} is necessary: for example, if $\insA$ is the ring of all algebraic integers, $\star_P:=s_{\Max(\insA)\setminus\{P\}}$ and $\mathscr{D}:=\{\star_P\mid P\in\Max(\insA)\} \subseteq \SStarsp(\insA)$, then $\vee_{\mathscr{D}}$ is a semistar operation that closes $\mathbb{A}$ and thus closes every principal ideal of $\mathbb{A}$, while $\qspec^{\vee_{\mathscr{D}}}(D)=\{(0)\}$, hence $\vee_{\mathscr{D}}$ is not quasi-spectral.  (See  \cite[Example 4.4]{FiFoSp-4} for more details.)

\smallskip
Lemma \ref{prop:infsup}\ref{prop:infsup:sup} provides useful information on the supremum of a family of spectral semistar operations of  finite type, allowing one  to prove that the space of all stable semistar operations of finite type is spectral. The proof of the following theorem follows closely the one of Theorem \ref{prop:insfinss-spectral}.

\begin{teor}\label{stabili}{\rm (cf. \cite[Theorem 4.5]{FiFoSp-4})}
	Let $D$ be an integral domain. Then, $\SStarstabft(D)$ is a spectral space.
\end{teor}

Stable semistar operations are closely related to the concept of \emph{localizing systems}, in the sense of Gabriel-Popescu (cf. for instance \cite[Chap. II]{BAC}, \cite{ga,c,la,s}). Recall that a localizing system on $D$ is a subset $\mathcal{F}$ of ideals of $D$ such that:
\begin{itemize}
	\item if $I\in\mathcal{F}$ and  $J$ is an ideal of $D$ such that $I\subseteq J$, then $J\in\mathcal{F}$;
	\item if $I\in\mathcal{F}$  and  $J$ is an ideal of $D$ such that, for each $i\in I$, $(J:_D iD)\in\mathcal{F}$, then $J\in\mathcal{F}$.
\end{itemize}
A localizing system $\mathcal{F}$ is said to be \emph{of finite type} if for each $I\in\mathcal{F}$ there exists a nonzero finitely generated ideal $J\in\mathcal{F}$ such that $J\subseteq I$. For instance, if $T$ is an overring of $R$, $\mathcal{F}(T):=\{I \mid I\text{~ideal of~}D,IT=T\}$ is a
localizing system of finite type, while, if $V$ is a valuation domain with a nonzero idempotent prime ideal $P$, then
$\hat{\mathcal{F}}( P): =\{I \mid I\text{~ideal of~} V\text{~and~}I\supseteq P\}$ is a localizing system of $V$ which is not of finite type
 \cite[Proposition 5.1.12 and Remark 5.1.13]{fohupa}. We denote by $\locsist(D)$ (respectively, $\locsistf(D)$) the set of all localizing systems (respectively, localizing systems of finite type) on $D$. We can introduce on these sets a natural topology, that we still call the \emph{Zariski topology}, whose subbasic open sets are the $\texttt{W}_I:=\{\mathcal{F}\in \locsist(D)\mid I\in\mathcal{F}\}$, as $I$ varies among the ideals in $D$.

\begin{teor} \label{locsist-stable} {\rm (cf.\!  \cite[Proposition 3.5, Proposition 4.1(5) and Corollary 4.6 ]{FiFoSp-4})}
	Let $D$ be an integral domain. The map $\boldsymbol{\lambda}:\locsist(D)\rightarrow \SStarstab(D)$ (respectively, the map $\boldsymbol{\lambda}_{{_f}}:\locsistf(D)\rightarrow \SStarstabft(D)$),  defined by $\mathcal{F}\mapsto\star_{\mathcal{F}}$, establishes a homeomorphism between spaces endowed with the Zariski topologies (respectively, the induced topologies from the Zariski topologies). In particular, by Theorem \ref{stabili}, $\locsistf(D)$ is a spectral space. 
\end{teor}

\section{The space of inverse-closed subsets of a spectral space}

Let $D$ be an integral domain. By the results in the previous sections, the spaces $\overr(D)$, $\SStarstabft(D)$ and $\SStarf(D)$ are spectral spaces. Since $\spec(D)$ can be embedded in each of these spaces, 
they can be seen as peculiar ``spectral extensions'' of $\spec(D)$. 

 In particular, in this section we focus on  the canonical embedding $  \spec(D)\hookrightarrow \SStarstabft(D)$, in order to generalize this spectral extension to arbitrary rings or to arbitrary spectral spaces. For this purpose, we need some preliminaries, including the notions and properties of (1.3).

\medskip

 We start by observing that 
 the   natural injection  $s:\spec(D)\rightarrow\SStarstabft(D)$, defined by $s( P):= s_{\{P\}} =  \wedge_{\{D_P\}}$, 
  is a topological embedding of topological (spectral) spaces (both endowed with the Zariski topology). Indeed, if 
  $J$ is a finitely gene\-ra\-ted ideal of $D$ and   $ \widetilde{\texttt{U}}_J := {\texttt{V}}_J \cap \SStarstabft(D) = \{\star\in \SStar(D)\mid 1\in J^\star\} \cap \SStarstabft(D)$ is a generic   subbasic open set of $\SStarstabft(D)$, then 
  $$
  s^{-1}(\widetilde{\texttt{U}}_J) = \{P\in \spec(D)  \mid 1\in JD_P \} = \texttt{D}(J)\,.
  $$

  \begin{oss}      The map $s:\spec(D)\rightarrow\SStarstabft(D)$ is the composition of the homeomorphism $\ell: \spec(D) \rightarrow \Loc(D)$, defined by $\ell({P}) := D_P$, for each $P \in \spec(D)$ and the topological embedding $\iota_{\mbox{\it\tiny\texttt{Loc}}}:  \Loc(D) \hookrightarrow \SStarstabft(D)$ (defined in Remark \ref{sp-st}(b)).  Note also that the homeomorphism $\ell$ induces an isomorphism of partially ordered sets (with the ordering induced by the topologies), however the ordering in $ \Loc(D) $, induced by the Zariski topology, is the opposite order of the set-theoretic inclusion.
    \end{oss}
   
     \smallskip

  Given a spectral space $X$, let $\boldsymbol{\mathcal{X}}(X) := \{ Y \subseteq X \mid   Y \neq \emptyset,\,  Y =\Cl^{\texttt{inv}}(Y)\}$. 
   If $X=\spec(R)$ for some ring $R$, we write for short $\xcal(R)$ instead of $\xcal(\spec(R))$.

 We define a \textit{Zariski topology on} $\boldsymbol{\mathcal{X}}(X)$ by taking, as subbasis of open sets, the sets of the form 
$$
\boldsymbol{\mathcal{U}}(\Omega):=\{Y\in \boldsymbol{\mathcal{X}}(X) \mid Y\subseteq \Omega \},$$
where $\Omega$ varies among the quasi-compact open subspaces of $X$. 
Note that the previous subbasis is in fact a basis, since $\boldsymbol{\mathcal{U}}(\Omega)\cap \boldsymbol{\mathcal{U}}(\Omega')=\boldsymbol{\mathcal{U}}(\Omega\cap \Omega')$ and $\Omega\cap \Omega'$ is a quasi-compact open subspace of $X$, for any pair $\Omega,\Omega'$ of quasi-compact open subspaces  of $X$.
Moreover, $\Omega \in \boldsymbol{\mathcal{U}}(\Omega)$, since a quasi-compact open subset $\Omega$ of $X$ is a closed   set  in the inverse topology of $X$. Note also that,  when $X=\spec( R)$, for some ring $R$, a generic basic open set of the Zariski topology on $\boldsymbol{\mathcal{X}}(R)$ is of the form
$$
\boldsymbol{\mathcal{U}}(\texttt{D}(J))=\{Y\in \boldsymbol{\mathcal{X}}(R) \mid Y\subseteq \texttt{D}(J) \}
$$
where $J$ is any finitely generated ideal of $R$.

\bigskip

The main result in this setting is the following, which provides a description of the space $ \boldsymbol{\mathcal{X}}(X)$ (see \cite{FiFoSp-1}).

\begin{teor} \label{embedding} 
Let $X$  be a spectral space.
\begin{itemize}
\item[\rm (1)] The space
 $ \boldsymbol{\mathcal{X}}(X)$, 
endowed with the Zariski topology, is a spectral space.
 
  \item[\rm (2)] Let $Y_1, \ Y_2 \in \boldsymbol{\mathcal{X}}(X)$. Then, 
$ Y_1 \subseteq Y_2$   if and only if   $ Y_1\leq _{\boldsymbol{\mathcal{X}}(X)} Y_2$.
 \item[\rm (3)] The canonical map $\varphi:X\rightarrow \boldsymbol{\mathcal{X}}(X)$, defined by
$
\varphi(x):=\{x\}^{\mbox{\rm\tiny\texttt{gen}}}
$, for each $x\in X$,   is  a spectral embedding  (which is also an order-preserving embedding  between ordered sets, with the ordering induced by the Zariski topologies). 
\item[\rm (4)]   $\boldsymbol{\mathcal{X}}(X)$ has a unique 
maximal point (i.e., $X$). 

 \item[\rm (5)] 
 Let $Z$  be another spectral space and  let $\varphi: X\rightarrow \boldsymbol{\mathcal{X}}(X)$
  be the spectral embedding  defined in (3).
Consider a spectral map  $\lambda: X \rightarrow Z$ satisfying the following condition:

\begin{itemize}
\item[ ] \emph{(\texttt{sup-completion})}  
{\sl 
For each nonempty quasi-compact subspace $Y$ of $X$, 
 there exists $z_Y:=\sup \{\lambda(y) \mid y \in Y \}$ (where {\rm sup} is taken with respect to the ordering induced by the topology of $Z$)  and if  \ $Y'$ is another nonempty quasi-compact subspace   of $X$, with 
 $ \mbox{\rm\texttt{Cl}}_X^{\mbox{\rm\tiny\texttt{inv}}}{(Y')} \neq  \mbox{\rm\texttt{Cl}}_X^{\mbox{\rm\tiny\texttt{inv}}}{(Y)}$, then  $z_{Y'} \neq z_Y$. 
 Moreover, if  $\mathcal W$  denotes the set of all nonempty quasi-compact open subspaces $\Omega$ of $X$, then
  $\mathscr{B} :=\{\{z_\Omega\}^{\mbox{\rm\tiny\texttt{gen}}} \mid \Omega \in \mathcal W \}$ is a subbasis for the open sets of $Z$.
  }
 \end{itemize}
Then, the following properties hold.
\begin{itemize}
\item[\rm (5.a)]
There exists a spectral embedding  $\boldsymbol{\lambda^{\sharp}}: \boldsymbol{\mathcal{X}}(X)\rightarrow Z$ such that $\boldsymbol{\lambda^{\sharp}} \circ \varphi = \lambda$.

\item[\rm (5.b)] If, furthermore, 
$z =  \sup_Z \{\lambda(x) \mid x \in \lambda^{-1}(\{z\}^{\mbox{\rm\tiny\texttt{gen}}})\}$ for each $z \in Z$,  then $\boldsymbol{\lambda^{\sharp}}: \boldsymbol{\mathcal{X}}(X)\rightarrow Z$ is the unique spectral embedding   (in fact, homeomorphism) such that $\boldsymbol{\lambda^{\sharp}} \circ \varphi = \lambda$.
\end{itemize} 
\end{itemize}
\end{teor}

Let $X$ be a spectral space and let $\boldsymbol{\hat{\mathcal{X}}}(X):=  \{ Y \subseteq X \mid  Y = {\Cl^{\mbox{\rm\tiny\texttt{inv}}}}(Y)\} =\boldsymbol{\ \mathcal{X}}(X) \cup \{\emptyset\}$. The techniques used for proving Theorem \ref{embedding}(1) allow also to show that $\boldsymbol{\hat{\mathcal{X}}}(X)$ (endowed with an obvious extension of the topology of 
$\boldsymbol{ \mathcal{X}}(X)$) is a spectral space.
Moreover, since $\ucal(\emptyset)=\{\emptyset\}$ is open in $\boldsymbol{\hat{\mathcal{X}}}(X)$, then we deduce that 
$\boldsymbol{\mathcal{X}}(X)$ is a closed (spectral) subspace of  $\boldsymbol{\hat{\mathcal{X}}}(X)$.

\medskip
 
As a consequence of the previous theorem, it is possible to compare the dimensions of $X$ and $\xcal(X)$ with the cardinality $|X|$ of the spectral space $X$ (see \cite{FiFoSp-1}). 

\begin{prop}\label{homeo}  
Let $X$ be a spectral space and let   $\varphi:X \rightarrow \boldsymbol{\mathcal{X}}(X)$   be  the topological embedding defined in Theorem \ref{embedding}(2). Then,  
\begin{itemize}
\item[\rm (1)] $\varphi(X)=   \boldsymbol{\mathcal{X}}(X)$ if and only if $(X,\leq)$ is linearly ordered.
\item[\rm (2)] $\dim(\xcal(X))=|X|-1\geq\dim(X)$. Moreover, in the finite dimensional case, $\dim(\xcal(X))$ $=\dim(X)$ if and only if $X$ is linearly ordered. 
\end{itemize}
\end{prop}

While the inequality $|X|-1\geq\dim(X)$ is sharp,  
the more non-comparable elements the set $X$ contains, the smaller $\dim(X)$ is  with respect to $|X|$. For example, if $X$ is homeomorphic to the prime spectrum of the direct product of $n+1$ fields, $n\geq 1$, then $\dim(X)=0$ while $|X|-1=n$.  

Furthermore, if $\dim(X)$ is not finite, then clearly $\dim(\xcal(X))=\dim(X)$, but we can easily choose $X$  to be not totally ordered. 

We also note that, if $\phi:X\longrightarrow Y$ is a spectral map of spectral space, the map $\xcal(\phi):\xcal(X)\longrightarrow\xcal(Y)$ defined by $\xcal(\phi)(C):=\phi(C)^{\mathrm{gen}}$ for every inverse-closed subset $C$ of $X$ is again a spectral map. It follows that the assignment $X\mapsto\xcal(X)$, $\phi\mapsto\xcal(\phi)$ is a (covariant) functor from the category of spectral spaces into itself (see
\cite{FiFoSp-1} 
  for details).

\medskip
We show next that  the map $\boldsymbol{\lambda^{\sharp}}: \boldsymbol{\mathcal{X}}(X)\rightarrow Z$ (Theorem \ref{embedding}(5.a)) is not unique. 
The following example shows in fact that it is possible that there exist two different spectral maps (with at most one non-injective) $\boldsymbol{\Lambda_1}, \boldsymbol{\Lambda_2}: \boldsymbol{\mathcal{X}}(X)\rightarrow Z$, $\boldsymbol{\Lambda_1} \neq \boldsymbol{\Lambda_2}$, such that $\boldsymbol{\Lambda_1}\circ \varphi = \lambda =\boldsymbol{\Lambda_2}\circ \varphi$.

\begin{ex}\label{ex-ext}
Consider the spectral space $X:=\{0,a,b,c\}$, with $0<a,b,c$ and $a,b,c$ not comparable. Let 
$\boldsymbol{\Lambda}:\boldsymbol{\mathcal{X}}(X)\rightarrow \boldsymbol{\mathcal{X}}(X)$ be the function defined by
$$
\boldsymbol{\Lambda}(C):=\begin{cases}
C & \mbox{ if } C\neq \{a,b\}^{\mbox{\rm\tiny\texttt{gen}}},\\
X & \mbox{ if } C=\{a,b\}^{\mbox{\rm\tiny\texttt{gen}}}.
\end{cases}
$$

The unique basic open set of $\boldsymbol{\mathcal{X}}(X)$ containing $\{a,b \}^{\mbox{\rm\tiny\texttt{gen}}}$ is \ 
 $\boldsymbol{\mathcal{U}}(\{a,b \}^{\mbox{\rm\tiny\texttt{gen}}})$, 
and clearly we have $\boldsymbol{\Lambda}^{-1}(\boldsymbol{\mathcal{U}}(\{a,b\}^{\mbox{\rm\tiny\texttt{gen}}}))=\boldsymbol{\mathcal{U}}(\{a\}^{\mbox{\rm\tiny\texttt{gen}}})\cup \boldsymbol{\mathcal{U}}(\{b\}^{\mbox{\rm\tiny\texttt{gen}}})$. 
For any other basic open set $\boldsymbol{\mathcal{U}} $ of $\boldsymbol{\mathcal{X}}(X)$, we have 
$\boldsymbol{\Lambda}^{-1}(\boldsymbol{\mathcal{U}})=\boldsymbol{\mathcal{U}} $.
 This shows that $\boldsymbol{\Lambda}$ is a nontrivial spectral map, $\boldsymbol{\Lambda} \neq {\texttt{id}}_{\boldsymbol{\mathcal{X}}(X)}$, such that $\boldsymbol{\Lambda}(\{x\}^{\mbox{\rm\tiny\texttt{gen}}})=\{x\}^{\mbox{\rm\tiny\texttt{gen}}}$, for each $x\in X$. 
\end{ex}

The following statement provides an explicit characterization of the space $\boldsymbol{\mathcal{X}}(X)$ and follows immediately from Theorem \ref{embedding}(5).

\begin{cor}
Let $\lambda:X\rightarrow Z$ be a spectral embedding of spectral spaces. Then, the following conditions are equivalent. 
\begin{enumerate}[\,\,\,\rm (i)]
\item $Z$ is a partially ordered set (under the ordering induced by the topology), for each $z \in Z$, 
$z =  \sup_Z \{\lambda(x) \mid x \in \lambda^{-1}(\{z\}^{\mbox{\rm\tiny\texttt{gen}}})\}$, and $\lambda$ satisfies  the condition \emph{(\texttt{sup-completion})}.
\item $Z$ is homeomorphic to $\boldsymbol{\mathcal{X}}(X)$, via a unique homeomorphism $\boldsymbol{\Lambda}:\boldsymbol{\mathcal{X}}(X)\rightarrow Z$ such that $\boldsymbol{\Lambda} \circ \varphi = \lambda$. 
\end{enumerate}
\end{cor}

\medskip
In the special case  where $X = \spec(D)$ for some  integral domain $D$, the spectral space $\boldsymbol{\mathcal{X}}(D) := \{ Y \subseteq \spec(D) \mid \emptyset \neq Y =\Cl^{\mbox{\rm\tiny\texttt{inv}}}(Y)\}$ can be interpretated in terms of  stable semistar operations of finite type (see \cite{FiFoSp-1}).

\begin{prop} \label{closed-inverse} 
Let $D$ be an integral domain and let $\boldsymbol{\mathcal{X}}(D) := \{ Y \subseteq \spec(D) \mid \emptyset \neq Y ={\mbox{\rm$\Cl$}}^{\mbox{\rm\tiny\texttt{inv}}}(Y)\}$. The map $\boldsymbol{s^\sharp}: \boldsymbol{\mathcal{X}}(D)\rightarrow \SStarstabft(D)$, defined by $\boldsymbol{s^\sharp}(Y) := s_Y$ for each $Y \in \boldsymbol{\mathcal{X}}(D)$, is a homeomorphism with inverse map $\Delta: \SStarstabft(D)\rightarrow  \boldsymbol{\mathcal{X}}(D)$,  defined by $\Delta(\star) := \mbox{\rm\qspec}^\star(D)$ for each $\star$ stable semistar operation of finite type on $D$.
Moreover,   if $\varphi:   \spec(D) \rightarrow \xcal(D)$ is the topological embedding defined in Theorem \ref{embedding}(3) and $s: \spec(D)\rightarrow \SStarstabft(D)$  is the topological embedding defined by $P \mapsto s_{\{P\}}$, for each prime ideal $P$ of $D$, then
$\boldsymbol{s^\sharp} \circ \varphi = s$.
 \end{prop}

 As a consequence of the previous proposition and Theorem \ref{embedding}(1) we reobtain immediately Theorem \ref{stabili}, that is, the space  of all stable semistar operations of finite type on an integral domain is a spectral space.

\section{A topological version of Hilbert's Nullstellensatz}

As a first application of the general construction considered in the previous section,  we give now a topological version of Hilbert's Nullstellensatz.

Given a ring $R$, consider  the set  
$\texttt{Rd}({R}) := \{I \mid I \mbox{ ideal of } R \mbox{ and } I = \mbox{rad}(I) \}$ of radical ideals of $R$ and, more generally, the set
 $\texttt{Id}({R}) := \{I \mid I \mbox{ ideal of } R \}$, endowed with the {\it hull-kernel topology}, defined by taking as a basis for the open sets the subsets
 $$
\boldsymbol{U}(x_1, x_2, \dots, x_n) := \{ I \in \texttt{Id}( R)  \mid x_i \notin I \mbox{ for some }  i,\ 1\leq i \leq n \}\,,
$$
where $x_1, x_2, \dots, x_n \in R$.    We denote by $\texttt{Id}({R})^{\texttt{hk}}$ (respectively, $\texttt{Rd}({R})^{\texttt{hk}}$) the set of all the ideals of $R$  (respectively, of all the radical ideals of $R$), endowed with the  hull-kernel topology (respectively, with the induced topology from the hull-kernel topology of  $\texttt{Id}({R})$). In this situation, the inclusion maps   $\spec({R})\subseteq \texttt{Rd}(R) \subseteq \texttt{Id}({R})$ become topological embeddings; in other words the hull-kernel topology induced  on $\spec({R})$ coincides with the Zariski topology.  
\medskip

For deepening the study of the topological space $ \texttt{Rd}({R})^{\texttt{hk}}$ we introduce an analogue, in the inverse topology, of the space $\xcal( R)$ (Section 4).
\medskip

Let $X$ be a spectral space and let $\chius{(Y)}$ denote the closure of a subspace $Y$ in the given topology of $X$.   For the sake of simplicity, we denote by $X^\prime$  the spectral space $X^{\texttt{inv}}$, i.e., the set $X$  
endowed with the inverse topology \cite[Proposition 8]{ho}. 
We set  $\boldsymbol{\mathcal{X}^\prime}(X):= \{ Y \subseteq X \mid Y \neq\emptyset, \ Y = \chius{(Y)}\}$ and, 
for each quasi-compact open subspaces $\Omega$ of $X$, we set $
\boldsymbol{\mathcal{U}^\prime}(\Omega):=\{Y\in \boldsymbol{\mathcal{X^\prime}}(X) \mid Y\cap\Omega =\emptyset \} =\ucal(\Omega') $, where $\Omega':= X \setminus \Omega$.

It is well known that $(X^{\texttt{inv}})^{\texttt{inv}}$ coincides with $X$ (with the given spectral topology)   \cite[Proposition 8]{ho} hence, {\sl mutatis mutandis}, we can now apply Theorem \ref{embedding},  since  $\boldsymbol{\mathcal{X}^\prime}(X) = \xcal(X')$, and we easily get the following.

\begin{prop} \label{inv-embedding} 
Let $X$ be a spectral space and let $ X^\prime:=X^{\mbox{\rm\tiny\texttt{inv}}}$. 
\begin{itemize}
\item[\rm (1)] The space
$ \boldsymbol{\mathcal{X}^\prime}(X) := \{ Y \subseteq X \mid Y \neq\emptyset, \ Y = \mbox{\rm $\Cl$}{(Y)}\}$ is a spectral space, when endowed with the topology, called \emph{the Zariski topology}, having as a basis of open sets, the sets of the form 
$\boldsymbol{\mathcal{U}^\prime}(\Omega)$,
where $\Omega$ varies among the quasi-compact open subspaces of $X$.
 
%
\item[\rm (2)] The canonical map $\varphi': X^\prime\rightarrow \boldsymbol{\mathcal{X}^\prime}(X)$, defined by
$
\varphi'(x):=\{x\}^{\mbox{\rm\tiny\texttt{sp}}}
$, for each $x\in X$, is  a spectral embedding between spectral spaces.
%
\end{itemize}
\end{prop}

 Suppose now that  $X:=\spec( R)$ is the prime spectrum of a commutative ring $R$, endowed with the Zariski topology. 
We recall that a basis of open sets of $X^{\texttt{inv}}$ is the collection of sets
$\{\texttt{V}(J) \mid J $  is a finitely generated ideal of $ R \}$
which makes $X^{\texttt{inv}}$ a spectral space  \cite[Proposition 8]{ho}.
\smallskip

\begin{oss} 
 With the notation introduced above,
 let  $\varphi': X' = \spec(R)^{\mbox{\rm\tiny\texttt{inv}}} \hookrightarrow
  \xcal(X')^{\mbox{\rm\tiny\texttt{zar}}}=
  \xcal^{\boldsymbol{\prime}}(X)^{\mbox{\rm\tiny\texttt{zar}}}$ 
 be the canonical topological embedding defined by $\varphi'(x):=\{x\}^{\mbox{\rm\tiny\texttt{sp}}}$. Then, it is easy to see that the map  $\psi := (\varphi')^{\mbox{\rm\tiny\texttt{inv}}}: X = 
 (\spec(R)^{\mbox{\rm\tiny\texttt{inv}}})^{\mbox{\rm\tiny\texttt{inv}}} 
 \hookrightarrow 
 \xcal^{\boldsymbol{\prime}}(X)^{\mbox{\rm\tiny\texttt{inv}}}$ 
 defined by $\psi(x):=\{x\}^{\mbox{\rm\tiny\texttt{gen}}}$  is a topological embedding (acting  like $\varphi$ as a set-theoretic map).
%
%
\end{oss}

The next result provides a topological version\ of Hilbert Nullstellensatz (see \cite{FiFoSp-2}).


\begin{teor}\label{prop:rad-x1} 
Let $R$ be a ring and let $ \boldsymbol{\mathcal{X}^\prime}( R) :=  \boldsymbol{\mathcal{X}^\prime}(\spec( R))$ be the spectral space of the non-empty Zariski closed subspaces of $\spec( R)$ (Proposition \ref{inv-embedding}). We can also consider the space $ \boldsymbol{\mathcal{X}^\prime}( R)$ as a spectral space endowed with the inverse topology  \cite[Proposition 8]{ho}.   Then,  for each $ C \in  \boldsymbol{\mathcal{X}^\prime}( R)$, the map: 
$$\mathscr{J}: \boldsymbol{\mathcal{X}^\prime}({R})^{\mbox{\rm\tiny\texttt{inv}}} \rightarrow \mbox{\rm\texttt{Rd}}({R})^{\mbox{\rm\tiny\texttt{hk}}} \mbox{ defined by }\boldsymbol{\mathscr{J}}( C) := \bigcap\{P\in\spec(R) \mid P\in C\}, $$
 is  a homeomorphism.
\end{teor}

\medskip

Related to the previous Theorem \ref{prop:rad-x1}, it is possible to prove, with a standard argument based on Theorem \ref{spec-charact}, that the set of all ideals of a ring is also a spectral space.  More precisely:
 \begin{prop} {\rm (cf. \cite{FiFoSp-2})}
 Let $\mbox{\rm\texttt{Id}}({R})$ be the space of all ideals of a ring $R$, endowed with the hull-kernel topology. Then,  $\mbox{\rm\texttt{Id}}({R})$  is  a spectral space, having $\mbox{\rm\texttt{Rd}}( R)$ (endowed with the hull-kernel topology) as a spectral subspace.
\end{prop}



The following Hasse diagram summarizes some of the results proved above.
\begin{equation*}
\begin{tikzcd}
& & \mbox{\rm\texttt{Id}}(R)^{\mbox{\rm\tiny\texttt{hk}}} \\
\xcal'(\spec(R))^{\mbox{\rm\tiny\texttt{inv}}} & \simeq &  \mbox{\rm\texttt{Rd}}({R})^{\mbox{\rm\tiny\texttt{hk}}}\arrow[hook]{u} & & \xcal(\spec(R))^{\mbox{\rm\tiny\texttt{zar}}}\\
\spec(R)^{\mbox{\rm\tiny\texttt{zar}}} \arrow[hook]{u} \arrow[equals]{rr} & & \spec(R)^{\mbox{\rm\tiny\texttt{hk}}} \arrow[hook]{u}\arrow[equals]{rr} & & \spec(R)^{\mbox{\rm\tiny\texttt{zar}}}\arrow[hook]{u}
\end{tikzcd}
\end{equation*}



\section{The  space of {\rm \texttt{eab}} semistar operations of finite type}

In the present section, we give another application of  Theorem \ref{embedding}.
 More precisely, we   apply the construction of the space $\xcal (X)$ to the case of the Riemann-Zariski spectral space $X:= \Zar(D)$ of all valuation overrings of an integral domain $D$ (endowed with the Zariski topology, see (1.2)).

Let $\star$ be a semistar operation on an integral domain $D$. We say that $\star$ is an \emph{\texttt{eab} semistar operation} (respectively, an \emph{\texttt{ab} semistar operation}) if, for every $F,G,H\in\f(D)$ (respectively, for every $F\in\f(D)$, $G,H\in\FF(D)$) the inclusion $(FG)^\star\subseteq(FH)^\star$ implies $G^\star\subseteq H^\star$. Note that, if $\star$ is \texttt{eab}, then $\stf$ is also \texttt{eab}, since $\star$ and $\stf$ agree on finitely generated fractional ideals. The concepts of \texttt{eab} and \texttt{ab} operations coincide on finite-type operations, but not in general    \cite{fo-lo-2009,fo-lo-ma}.

It is easy to see that a {\it valuative semistar operation}, i.e., a semistar operation of the type $\wedge_{\calbW}$, where $\calbW \subseteq \Zar(D)$, is an \texttt{eab} semistar operation. In particular, the $\mbox{\it\texttt{b}}$-operation, where $\mbox{\it\texttt{b}}  := \wedge_{\Zar(D)}$, is  an \texttt{eab} semistar operation of finite type, since  $\Zar(D)$ is quasi-compact (Corollary \ref{qcover}).

\medskip

To every semistar operation $\star \in \SStar(D)$ we can associate a map $\sta$ defined by
\begin{equation*}
F^{\sta}:=\bigcup\{((FG)^\star:G^\star)\mid G\in\f(D)\}
\end{equation*}
for every $F \in\f(D)$, and then we can extend it to arbitrary $D$-modules $E \in \FF(D)$ 
by setting  $E^{\sta} :=\bigcup\{F^{\sta}\mid F\subseteq E, \ F\in\f(D)\}$. The map $\sta$ is always an \texttt{eab} semistar operation of finite type on $D$. Moreover,   $\star= \sta$ if and only if $\star$ is an \texttt{eab} semistar  operation of finite type  and, if $\star$ is an \texttt{eab} semistar operation, then $\sta =\stf$    \cite[Proposition 4.5]{folo-2001}.

\begin{oss} (a)
Let $T$ be an overring of $D$, and let $\star_T$ be a semistar operation on $T$. Then, we can define a semistar operation $\star$ on $D$ by $\star:=\star_T\circ\wedge_{\{T\}}$, i.e., $E^\star:=(ET)^{\star_T}$ for every $E\in\FF(D)$. If now $F\in\f(T)$, then
$$
\begin{array}{rl}
F^{\star_a} =& \hskip -6pt   \bigcup \{((FG)^\star:G^\star) \mid  G\in\f(D) \} =
           \bigcup \{((FGT)^{\star_T}:(GT)^{\star_T}) \mid  G\in\f(D)\}=\\
=& \hskip -6pt  \bigcup\{ ((FTH)^{\star_T}:H^{\star_T}) \mid  H\in\f(T)\}=(FT)^{(\star_T)_a}=F^{(\star_T)_a}.
\end{array}
$$
Hence, for every $E\in\FF(D)$, $E^{\star_a}=(ET)^{(\star_T)_a}$, that is, $\star_a=(\star_T)_a\circ\wedge_{\{T\}}$.

(b) W. Krull  only considered the concept of an ``\texttt{a}rithmetisch
\texttt{b}rauchbar'' operation(for short \texttt{ab}-operation, as above)  \cite{Krull:1936}. He did not
consider the concept of ``\texttt{e}ndlich \texttt{a}rithmetisch \texttt{b}rauchbar'' operation   (or, more simply, \texttt{eab}-operation as above).
This concept stems from the original version of Gilmer's book \cite{gi-1968}. 

(c) Denote by $\SStarval(D)$ (respectively,  $\SStareab(D)$; $\SStareabf(D)$) the set of valutative (respectively, \texttt{eab}; \texttt{eab}   of finite type) semistar operations on $D$. Every valutative operation is \texttt{eab}, but not every \texttt{eab} operation is valutative; however, the two definitions agree on finite-type operations, i.e.,
$$
\SStareab(D) \cap \SStarf(D) =: \SStareabf(D) = \SStarval(D) \cap \SStarf(D)\,,
$$
(see, for instance, \cite[Corollary 5.2]{folo-2001}). 
A similar  relationship holds between spectral and stable semistar operations, with the valutative operations corresponding to the spectral ones and the \texttt{eab} operations to the stable ones, i.e., every spectral semistar operation is stable but not every stable semistar operation is spectral, however  $
\SStarspf(D)=  \SStarstab(D) \cap \SStarf(D)=\SStarstabft(D)$ (Remark \ref{sp-st}(a)).

Recall also that there are examples of \texttt{eab} semistar operations which are quasi-spectral but not valutative \cite[Example 15]{fo-lo-2009}.
\end{oss}

It is not hard to prove the following statement, which is a companion to Proposition \ref{prop:retraction}.

\begin{prop} {\rm   (cf. \cite[Proposition 5.2]{FiFoSp-4})}
Let $D$ be an integral domain and let $\Phi_{\mbox{\it\tiny\texttt{a}}}:\SStar(D)$ $\rightarrow\SStar(D)$ be the map defined by  $\star\mapsto\sta$. Then:
	\begin{enumerate} 
		\item[\rm(1)] The image of $\Phi_{\mbox{\it\tiny\texttt{a}}}$ coincides with $\SStareabf(D)$.
		\item[\rm(2)] The map $\Phi_{\mbox{\it\tiny\texttt{a}}}$ is continuous in the Zariski topology.
		\item[\rm(3)] The map $\Phi_{\mbox{\it\tiny\texttt{a}}}$ is  a topological retraction of $\SStar(D)$ onto $\SStareabf(D)$.
	\end{enumerate}
\end{prop}

The relation between valutative operations and subsets of $\Zar(D)$ behave very similarly to the relation between spectral operations and subsets of $\spec(D)$  (Proposition \ref{spettrali}).

\begin{prop}\label{zar-inv}
	 Let $D$ be an integral domain and let $Y$ and $Z$ be two nonempty subsets of $\mbox{\rm\Zar}(D)$.
	Then, the following statements hold. 
	\begin{enumerate}
			\item[\rm(1)]  $\wedge_Y=\wedge_Z$ if and only if $Y^{\mbox{\rm\tiny\texttt{gen}}}=Z^{\mbox{\rm\tiny\texttt{gen}}}$. 		\item[\rm(2)]  $\wedge_Y$ is of finite type if and only if $Y$ is quasi-compact\;  \cite[Proposition 4.5]{FiSp}.

		\item[\rm(3)]    $(\wedge_Y)_{\!{_f}}=(\wedge_Z)_{\!{_f}}$ if and only if 
		 $\mbox{\rm{\texttt{Cl}}}^{\mbox{\rm\tiny\texttt{inv}}}(Y)=
		 \mbox{\rm{\texttt{Cl}}}^{\mbox{\rm\tiny\texttt{inv}}}(Z)$ \; \cite[Theorem 4.9]{fifolo2}.
		 \item[\rm(4)]    $(\wedge_Y)_{\!{_f}}= \wedge_{\mbox{\rm\footnotesize{\texttt{Cl}}}^{\mbox{\rm\tiny\texttt{inv}}}(Y)}$ \; \cite[Corollary 4.17]{fifolo2}.
	\end{enumerate}
\end{prop}
Note that $Y^{\mbox{\rm\tiny\texttt{gen}}} = \{ÊV \in \mbox{\rm\Zar}(D) \mid V \supseteq V_0, \mbox{ for some } V_0 \in Y \}$.
For the statement (1), assume first that $\wedge_Y=\wedge_Z$. Let $V$ be a valuation domain such that $V\in Y^{\mbox{\rm\tiny\texttt{gen}}}\setminus Z^{\mbox{\rm\tiny\texttt{gen}}}$. 
Then, for any $W\in Z$, we can pick an element $x_W\in W\setminus V$. 
It follows that $I:=(x_W^{-1} \mid W\in Z)\subseteq M_V$, where $M_V$ is the maximal ideal of $V$. 
Thus, if $V_0\in Y$ is such that $V_0\subseteq V$ (such a $V_0$ exists since $V\in Y^{\mbox{\rm\tiny\texttt{gen}}}$),
 we have $I V_0\subseteq M_{V_0}$ and, in particular, 
$1\notin I^{\wedge_Y}$. 
On the other hand, clearly $1\in I^{\wedge_Z}$, a contradiction.  The converse it is straightforward since, for each $Y \subseteq 
\mbox{\rm\Zar}(D)$, $\wedge_Y = \wedge_{Y^{\mbox{\rm\tiny\texttt{gen}}}}$. 
\medskip

\begin{oss}
Since $\mbox{\it \texttt{b}}=\wedge_{\Zar(D)}$ is a semistar operation of finite type (and this can be proved completely independently from the topological point of view, see \cite[Proposition 6.8.2]{swhu} and \cite[Remark 4.6]{FiSp}), from Proposition \ref{zar-inv} we get a new proof   of the fact that $\Zar(D)$ is a quasi-compact space (this is a special case of Zariski's theorem \cite[Theorem 40, page 113]{zs}).
\end{oss}

 The embedding $\iota:\overr(D)\rightarrow\SStarf(D)$ (Proposition \ref{immersione}) restricts to an embedding
  $\Zar(D)\hookrightarrow \SStareabf(D)$, while the image of the restriction $\pi|_{\SStareabf(D)}$ of the canonical map $\pi:\SStarf(D)\rightarrow \overr(D)$ (defined by $\star\mapsto D^\star$) coincides with $\overric(D)$, i.e., with  the space of the overrings of $D$ that are integrally closed in $K$ (since, by a  well known Krull's theorem, every integrally closed ring can be represented as an intersection of valuation rings   \cite[Theorem 6, page 15]{zs}).

\medskip

Using the $\mbox{\it \texttt{b}}$-operation, we can introduce a general version of the classical Kronecker  fun\-ction ring, introduced by L. Kronecker in the case of Dedekind domains. Let $\X$ be is an indeterminate over $D$ and let $\boldsymbol{c}(h)$ be the content of a polynomial $h \in D[\X]$ (i.e., the ideal of $D$ generated by the coefficients of $h$).  Then,  we set: 

$$
\begin{array}{rl}
\Kr(D):=\mbox{Kr}(D,\mbox{\it \texttt{b}}) :=&  \hskip -7pt\{ f/g  \mid \, f,g \in D[\X], \ g 
\neq 0, 
\; 
\mbox{ 
	with } \boldsymbol{c}(f)^{\mbox{\it \tiny\texttt{b}}}
\subseteq \boldsymbol{c}(g)^{\mbox{\it \tiny\texttt{b}}} \,\} \\
=& \hskip -7pt\bigcap\{ V(\X) \mid V \in \Zar(D)\},
\end{array}
$$
 where $V(\X)$ denotes the   Gaussian (or trivial) extension of $V$ to $K(\X)$, i.e., $V(\X):= V[\X]_{(MV[\X])}$.
This is a B\'ezout domain with quotient field $K(\X)$, called {\it  the $b$-Kronecker function ring  of $D$} (see  \cite[Definition  3.2, Corollary 3.4(2) and Theorem 5.1]{folo-2001}, \cite[Theorem 14]{folo-2006} and \cite[Theorem 32.11]{gi}).  
It follows immediately that the localization map $\spec(\Kr(D))\longrightarrow \Zar(\Kr(D))$ (defined by $P\mapsto\Kr(D)_P$) is actually an homeomorphism. Moreover, the map $\Psi:\Zar(D)\longrightarrow\Zar(\Kr(D))$ (defined by  $V\mapsto V(\X)$) is a homeomorphism  \cite[Propositions 3.1 and 3.3]{fifolo2}, so that $\spec(\Kr(D))$ realizes $\Zar(D)$ as a spectral space  \cite[Theorem 2]{dofo-86}.

In particular,   the homeomorphism (and so the isomorphism of partially ordered sets) that we denote by $\theta$, from  $\spec(\Kr(D))$  to $\Zar(D)$ induces a 1-1 correspondence
 $\boldsymbol{\Theta}_{\!{_0}}$ between the set  $\{ Y\subseteq\spec(\Kr(D)) \mid Y=Y^\downarrow\}$ (where  $Y^\downarrow:= \{z \in \spec(\Kr(D))\mid z \leq  y,$ for some $y \in Y\} =  {Y}^{\texttt{gen}}$) 
 and the set   $\{ \calbW \subseteq\Zar(D) \mid \calbW = \calbW^{\uparrow}\} $  
 (where  $\calbW^{\uparrow}:= \{W' \in \Zar(D) \mid W' \supseteq W,$ for some $W \in \calbW\} =  {\calbW}^{\texttt{gen}}$). 
 Therefore $\boldsymbol{\Theta}_{\!{_0}}$  induces a bijection $\boldsymbol{\Theta}: \SStarsp(\Kr(D)) \rightarrow \SStarval(D)  $ defined by  $\boldsymbol{\Theta}(s_{Y}):=\wedge_{\boldsymbol{\Theta}_{\!{_0}}(Y)}$, where  ${\boldsymbol{\Theta}_{\!{_0}}(Y)} = \{ V \in \Zar(D) \mid M(\X) \cap \Kr(D) \in Y\} =: \calbV(Y)$ and $M(\X)$ is the maximal ideal of $V(\X)$. 
 
 \begin{teor} \label{eab-op} {\rm  (cf. \cite[Theorem 5.11]{FiFoSp-4})}
	Let $D$ be an integral domain. Then, the bijection $\boldsymbol{\Theta}$, restricted to $\SStarf(D)$, induces a homeomorphism between $\SStarstabft(\Kr(D))$  and $\SStareabf(D)$. In particular, $\SStareabf(D)$ is a spectral space.
\end{teor}

Another interpretation of the previous theorem can be given by considering the spectral space $\xcal(X)$, when $X$ coincides with $\Zar(D)$. This point of view  sheds new light on the analogies between the spectral spaces  $\SStarstabft(D) \ (= \SStarspf(D)$, by Remark \ref{sp-st}(a))  and $\SStareabf(D)$, after recalling that $\xcal(D) :=\xcal(\spec(D))$ is canonically homeomorphic to $\SStarstabft(D)$ (Proposition \ref{closed-inverse}).

\begin{cor} \label{eab-spectral}
Let $D$ be an integral domain.
 The map 
 $$
 \Lambda:  \xcal(\mbox{\rm\Zar}(D)) \rightarrow \SStareabf(D), \mbox{ defined  by } \Lambda(\calbY):=   \wedge_{\calbY}\,,
 $$
  for each inverse-closed subset $\calbY$ of $\mbox{\rm\Zar}(D)$, is a homeomorphism.
\end{cor}
\begin{proof} {\sl (Sketch)} The proof is based on the following key facts.  The space   $\xcal(\Zar(D))$ is canonically homeomorphic to 
$\xcal(\Kr(D))$ \cite{FiFoSp-1}.  
 By Proposition \ref{closed-inverse},  $\xcal(\Kr(D)) \simeq \SStarstabft(\Kr(D)) \ (= \SStarspf(\Kr(D)))$ 
and finally that the  map $\boldsymbol{\Theta}_{\!\mbox{\it \tiny\texttt{f}}}$,
 restriction of  $\boldsymbol{\Theta}$ to ${\SStarspf(\Kr(D))}$,
 from $\SStarspf(\Kr(D))$ onto $\SStareabf(D)$, 
is a homeomorphism (for more details   \cite[Theorem 5.11(2)]{FiFoSp-4}).
\end{proof}

\bigskip


The following Hasse diagram summarizes the topological embeddings of some of the spaces considered in the present paper. All spaces are spectral, except possibly the three spaces denoted with {\Tiny {$^{(\blacklozenge)}$}}.
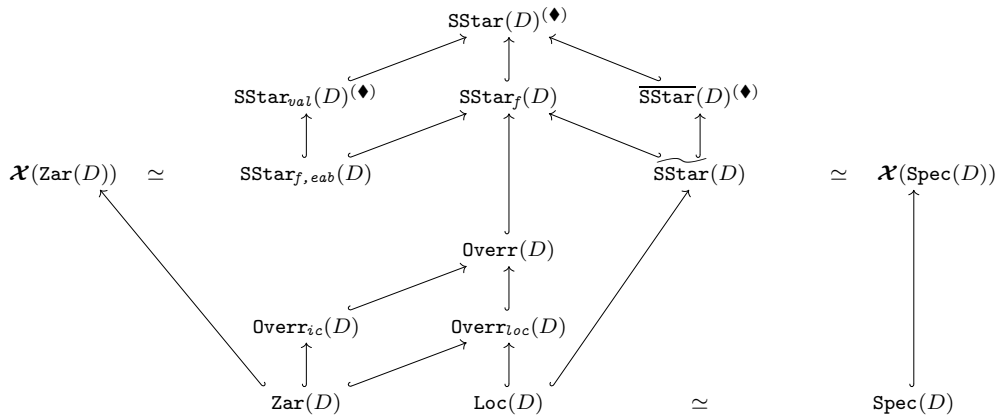
\begin{equation*}
{\Small
\begin{tikzcd}
& & \SStar(D)^{(\blacklozenge)}\arrow[hookleftarrow]{dr}\\
& \SStarval(D)^{(\blacklozenge)}\arrow[hook]{ur} & \arrow[hook]{u}\SStarf(D)\arrow[hookleftarrow]{dr} & \SStarstab(D)^{(\blacklozenge)}\\
\arrow[hookleftarrow,end anchor=north west]{dddr}\xcal(\Zar(D))~~~\simeq & \arrow[hook]{u}\SStareabf(D)\arrow[hook]{ur} & & \SStarstabft(D)\arrow[hook]{u} & \simeq~~~\xcal(\spec(D))\\
& & \arrow[hook]{uu}\overr(D)\\
& \overric(D)\arrow[hook]{ur} & \overrloc(D)\arrow[hook]{u}\\
& \arrow[hook]{u}\Zar(D)\arrow[hook]{ur} & \arrow[hook]{u}\Loc(D)\arrow[hook,start anchor=north east]{uuur} & \simeq & \spec(D)\arrow[hook]{uuu}\\
\end{tikzcd}
}
\end{equation*}




\begin{thebibliography}{12}


 



\bibitem{an-overrings} D.D. Anderson, Star-operations induced by overrings, \emph{Comm. Algebra} {\bf 16} (1988), 2535--2553.

\bibitem{ac} D.D. Anderson and S.J. Cook, Two star operations and their induced lattices, \emph{Comm. Algebra} {\bf 28} (2000), 2461--2475.

\bibitem{an} D.F. Anderson and D.D. Anderson, Examples of star operations on integral domains, \emph{Comm. Algebra} {\bf 18} (1990), 1621--1643.



\bibitem{am} M. F. Atiyah and I. G. Macdonald, Introduction to commutative algebra,
Addison-Wesley, Reading, 1969.

\bibitem{badofo} V. Barucci, D. Dobbs, M. Fontana, Conducive integral domains as pullbacks, \emph{Manuscripta Math.} \textbf{54}, 261--
277 (1986).

\bibitem{BAC}
N. Bourbaki, Alg\`ebre Commutative, Chap. 1-2, Hermann, Paris, 1961.

\bibitem{c}
P.-J. Cahen,  Commutative torsion theory, \emph{Trans. Amer. Math. Soc.} {\bf 184} (1973), 73--85.

\bibitem{calota} P. J. Cahen, K. A. Loper, and F. Tartarone, Integer-valued polynomials and Pr\"ufer $v$-multiplication domains, \textit{J. Algebra} \textbf{226} (2000), 765--787.



\bibitem{ch-1955}
C. Chevalley, Sur la th\'eorie des vari\'et\'es alg\'ebriques, \emph {Nagoya Math. J.} {\bf 8} (1955), 1--43.



\bibitem{ch}
C. Chevalley and H. Cartan, {Sch\'emas normaux; morphismes; ensembles constructibles,} \emph{S\'eminaire Henri Cartan}  {\bf 8} (1955-1956), Exp. No. 7, 1--10.


\bibitem{dofefo-87}  D. Dobbs, R. Fedder and M. Fontana, Abstract Riemann surfaces of integral domains
and spectral spaces. \emph{Ann. Mat. Pura Appl.} {\bf 148} (1987), 101--115.

\bibitem{dofo-86} D. Dobbs and M. Fontana, Kronecker function rings and abstract Riemann surfaces.
\emph{J. Algebra} {\bf 99} (1986),  263--274.


 

\bibitem{ep-12} N. Epstein, A guide to closure operations in commutative algebra. {\sl Progress in Commutative Algebra 2}, 1--37, Walter de Gruyter, Berlin, 2012.

\bibitem{ep-15} N. Epstein, Semistar operations and standard closure operations, \emph{Comm. Algebra} {\bf 43} (2015), 325--336. 

\bibitem{Fi} C. A. Finocchiaro, Spectral spaces and ultrafilters, \emph{Comm. Algebra} {\bf 42} (2014), 1496--1508.

\bibitem{fifolo1} C. A. Finocchiaro, M. Fontana, and K. A. Loper, Ultrafilter and constructible topologies on spaces of valuation domains, \emph{Comm. Algebra}  {\bf 41} (2013),  1825--1835.

\bibitem{fifolo2} C. A. Finocchiaro, M. Fontana, and K. A. Loper, The constructible topology on spaces of valuation domains, \emph{Trans. Am. Math. Soc.} {\bf } (2013), 6199--6216.


 \bibitem{FiFoSp-4} C. A. Finocchiaro, M. Fontana, and D. Spirito, Spectral spaces of semistar operations, 2015 (in preparation).
 
\bibitem{FiFoSp-1} C. A. Finocchiaro, M. Fontana, and D. Spirito,  The space of inverse-closed subsets of a spectral space, 2015 (in preparation).


%

 \bibitem{FiFoSp-2} C. A. Finocchiaro, M. Fontana, and D. Spirito,  On a topological version of Hilbert's Nullstellensatz,  2015 (in preparation). 



\bibitem{FiSp} C. A. Finocchiaro and D. Spirito, Some topological considerations on semistar operations, \emph{J. Algebra}, {\bf 409}(2014), 199--218.

\bibitem{fohu} M. Fontana and J. Huckaba, Localizing systems and semistar operations, in ``Non-Noetherian Commutative Ring Theory'' (Scott T.
Chapman and Sarah Glaz, eds.), Kluwer Academic Publishers, 2000,
pp. 169--198.

\bibitem{fohupa} M. Fontana, J. Huckaba, and I. Papick, Pr\"ufer domains, M. Dekker, New York, 1997.




\bibitem{folo-2001} M. Fontana and K.A. Loper,  Kronecker function rings: a general approach, in ``Ideal theoretic methods in commutative algebra'' (Columbia, MO, 1999), 189--205, Lecture Notes in Pure and Appl. Math., {\bf 220}, Dekker, New York, 2001.

\bibitem{fo-lo-2003} M. Fontana and K.A. Loper, Nagata rings, Kronecker function rings and related
semistar operations, \emph{Comm. Algebra} {\bf 31} (2003), 4775--4805

\bibitem{folo-2006} M. Fontana and K.A. Loper, An historical overview of Kronecker function rings, Nagata rings, and related star and semistar operations,  in 
{\sl ``Multiplicative Ideal Theory in Commutative Algebra:
A Tribute to the Work of Robert Gilmer''} (J.W. Brewer, S. Glaz, W. Heinzer, and B. Olberding Editors), 169--187, Springer, New York 2006.

\bibitem{fo-lo-2008}  M. Fontana and  K.A. Loper, The patch topology and the ultrafilter topology on the prime spectrum of a commutative ring, \emph{ Comm. Algebra} {\bf 36} (2008), 2917--2922.


\bibitem{fo-lo-2009}  M. Fontana and  K. A. Loper, Cancellation properties in ideal systems: a classification of e.a.b. semistar operations, \emph{J. Pure Appl. Algebra} {\bf 213} (2009), 2095--2103.

\bibitem{fo-lo-ma}
M. Fontana, K. A. Loper, and R. Matsuda,  Cancellation properties in ideal systems: an  e.a.b. not  a.b. star operation, \emph{AJSE (Arabian Journal for Science and Engineering)--Mathematics},  {\bf 35} (2010), 45--49.

\bibitem{ga} 
P. Gabriel, 
La localisation dans les anneaux non commutatifs, in {\it S\'eminaire Dubreil (sous la direction de P. Dubreil, M.-L. Dubreil-Jacotin, C. Pisot). Alg\`ebre et th\'eorie des nombres,} 13 no. 1, 1959-1960, Expos\'e No. 2, 35 p.  

\bibitem{gi-1968}
R. Gilmer, Multiplicative Ideal Theory, Vol. I \& II,  Queen's Papers in Pure and Applied Mathematics, Kingston, Ontario, Canada, 1968.
\bibitem{gi}
R. Gilmer, Multiplicative Ideal Theory, M. Dekker, New York, 1972.


\bibitem{EGA}
A. Grothendieck and J. Dieudonn\'e,  \'El\'ements de G\'eom\'etrie 
Alg\'ebrique I, IHES 1960; Springer, Berlin, 1970.

\bibitem{heinz-roit} W. Heinzer and M. Roitman, Well-centered overrings of an integral domain, \emph{J. Algebra}, \textbf{272}(2004), no.2, 435--455

\bibitem{ho} M. Hochster, Prime ideal structure in commutative rings, \emph{Trans. Amer. Math. Soc.} {\bf 142} (1969),   43--60.

\bibitem{hochster} M. Hochster and C. Huneke, Tight closure, invariant theory, and the {B}rian\c con-{S}koda theorem, \emph{J. Amer. Math. Soc.} \textbf{3} (1990), no.1, 31--116

\bibitem{hk}
O. Heubo-Kwegna, Kronecker function rings of transcendental field extensions, \emph{Comm. Algebra} {\bf 38} (2010),  2701--2719.
\bibitem{hu}  J. Huckaba, Commutative rings with zero divisors, M. Dekker, New York, 1988.

\bibitem{swhu} C. Huneke and I. Swanson, Integral Closure of Ideals, Rings, and Modules, London Math. Soc. Lecture Note Ser., vol. 336, Cambridge University Press, Cambridge, 2006.



\bibitem{Krull:1935}   W. Krull, Idealtheorie, Springer, Berlin, 1935 (Second Edition, 1968).


  \bibitem{Krull:1936}   W. Krull, Beitr\"age zur Arithmetik kommutativer
    Integrit\"atsbereiche, \rm I - II. \emph{Math.  Z.}  {\bf 41} (1936),
   545--577; $\,$ 665--679.
   
\bibitem{Krull} W. Krull,
Gesammelte Abhandlungen / Collected Papers,
Hrsg. v. Paulo Ribenboim,  Walter de Gruyter, Berlin,
1999. 

   \bibitem{je} T. Jech, Set Theory, Springer, New York, 1997 (First Edition, Academic Press, 1978).


\bibitem{la} J. Lambek, Torsion theories, additive semantics, and rings of quotients, Lecture Notes in Math., vol. 177, Springer-Verlag, Berlin and New York, 1971.



\bibitem{ma}
P. Maroscia, Sur les ann\'eaux de dimension z\'ero, \emph{Rend. Acc. Naz. Lincei }{\bf 56} (1974), 451--459.

 \bibitem{OM2} A. Okabe and R. Matsuda,  Semistar operations on
   integral domains, \emph{Math.
   J. Toyama Univ.} {\bf 17} (1994), 1--21.




\bibitem{oli1}
J.-P. Olivier, Anneaux absolument plats universels et \'epimorphismes \`a buts r\'eduits, \emph{ S\'em P. Samuel, Alg\`ebre Commutative,  Ann\'ee 1967/68,} Ex. N. 6.

\bibitem{oli2}
J.-P. Olivier, Anneaux absolument plats universels et \'epimorphismes d'anneaux, \emph{C.R. Acad. Sci. Paris} {\bf 266} (1968), 317--318.



 \bibitem{sch-tr}
 N. Schwartz  and M. Tressl, Elementary properties of minimal and maximal points in Zariski spectra, \emph{J. Algebra} {\bf 323} (2010), 698--728.



\bibitem{s}
B. Stenstr\"om, Rings and modules of quotients, Lecture Notes in Math., vol. 237,
Springer-Verlag, Berlin and New York, 1971.


\bibitem{va}
J.C. Vassilev, Structure on the set of closure operations of a commutative ring, \emph{J. Algebra} {\bf 321} (2009),   2737--2753.

\bibitem{za} O. Zariski, The compactness of the Riemann manifold of an abstract field of algebraic functions, \emph{Bull. Amer. Math. Soc} {\bf 50} (1944), 683--691.

\bibitem{zs}
O. Zariski and P. Samuel, Commutative Algebra, Volume II, Van Nostrand, Princeton, 1960.

\end{thebibliography}
\end{document}